\documentclass[a4paper,12pt]{scrartcl}
\usepackage{amsmath,amssymb,latexsym,amsthm,enumerate}
\usepackage{xcolor}
\usepackage{graphicx}
\usepackage{epstopdf}

\usepackage[T1]{fontenc}
\usepackage[utf8]{inputenc}
\usepackage[english]{babel}

\newcommand*{\wh}{\widehat}
\newcommand*{\wt}{\widetilde}
\newcommand*{\ol}{\overline}

\newcommand*{\eps}{\varepsilon}

\newcommand*{\N}{\mathbb{N}}
\newcommand*{\R}{\mathbb{R}}

\newcommand*{\Znn}{\bbN_0}

\newcommand*{\IR}{\mathbb{R}}

\newcommand*{\bbN}{\mathbb N}

\newcommand*{\bbR}{\mathbb R}

\newcommand*{\cF}{\mathcal{F}}

\newcommand*{\cM}{\mathcal{M}}

\newcommand*{\cO}{\mathcal{O}}

\newcommand*{\loc}{\mathrm{loc}}

\newcommand*{\Law}{\operatorname{Law}}
\newcommand*{\Var}{\operatorname{Var}}

\newcommand{\be}{\begin{eqnarray*}}
\newcommand{\ee}{\end{eqnarray*}}
\newcommand{\ben}{\begin{eqnarray}}
\newcommand{\een}{\end{eqnarray}}
\newcommand{\bi}{\begin{itemize}}
\newcommand{\ei}{\end{itemize}}

\newtheorem{theo}{Theorem}[section]

\newtheorem{lemma}[theo]{Lemma}
\newtheorem{propo}[theo]{Proposition}

\theoremstyle{definition}
\newtheorem{ex}[theo]{Example}

\newtheorem{remark}[theo]{Remark}

\newcounter{numpar}[section]

\title{Wasserstein convergence rates for random bit approximations of continuous Markov processes}

\author{Stefan Ankirchner\thanks{%
Stefan Ankirchner, Institute of Mathematics, University of Jena, Ernst-Abbe-Platz 2, 07745 Jena, Germany. \emph{Email:} s.ankirchner@uni-jena.de, \emph{Phone:} +49 (0)3641 946275.}
\and Thomas Kruse\thanks{%
Thomas Kruse, Institute of Mathematics, University of Gie{\ss}en, Arndtstr.~2, 35392 Gießen, Germany.
\emph{Email:} thomas.kruse@math.uni-giessen.de, \emph{Phone:} +49 (0)641 9932043.}
\and Mikhail Urusov\thanks{%
Mikhail Urusov, Faculty of Mathematics, University of Duisburg-Essen, Thea-Leymann-Str.~9, 45127 Essen, Germany.
\emph{Email:} mikhail.urusov@uni-due.de, \emph{Phone:} +49 (0)201 1837428.}}

\begin{document}

\maketitle

\begin{abstract}
We determine the convergence speed of a numerical scheme for approximating one-dimensional continuous strong Markov processes.
The scheme is based on the construction of certain
Markov chains whose laws can be embedded into the process with a sequence of stopping times. Under a mild condition on the process’ speed measure we prove that the approximating Markov chains converge at fixed times at the rate of $1/4$ with respect to every $p$-th Wasserstein distance. For the convergence of paths, we prove any rate strictly smaller than $1/4$.
Our results apply, in particular, to processes with irregular behavior
such as solutions of SDEs with irregular coefficients and processes with sticky points.


\end{abstract}

\section*{Introduction}\label{sec:intro}

In this article we analyze the convergence speed
in every $p$-th Wasserstein distance
of a numerical scheme, developed in \cite{aku2018cointossing}, that allows to approximate the law of any one-dimensional regular continuous strong Markov process (in the sense of Section~VII.3 in \cite{RY} or Section~V.7 in \cite{RogersWilliams}). 
In the following we refer to the latter processes as {\itshape general diffusions}. 

The set of general diffusions includes any one-dimensional stochastic process that can be described as a strong or weak solution of a time-homogeneous stochastic differential equation (SDE)
with possibly irregular coefficients.
There are, however, many general diffusions that cannot be characterized in terms of an SDE.
This is, in particular, true for diffusions with sticky features,
where a sticky point is located in the interior of the state space.
A related interesting phenomenon is slow (or sticky) reflection.
In this case the sticky point is located at the boundary of the state space.
Recent years have witnessed an increased interest
in diffusions with sticky features,
see
\cite{Bass2014},
\cite{CanCaglar2019},
\cite{ep2014},
\cite{FGV2016},
\cite{GV2017},
\cite{GV2018},
\cite{HajriCaglarArnaudon:17},
\cite{KSS2011},
\cite{Konarovskyi2017},
\cite{KvR2017}
and references therein.
Newly, diffusions with slow reflection
were applied in \cite{EberleZimmer2019}
to provide bounds (via sticky couplings)
for the distance between two
multidimensional diffusions with different drifts.
Diffusions with slow reflection also attracted interest in
economic theory, where such processes characterize
optimal continuation values
in dynamic principal-agent problems (see, e.g., \cite{zhu2012optimal} and \cite{piskorski2016optimal}).
We emphasize that the numerical schemes studied in this paper
are able to approximate general diffusions with irregularities
such as, e.g., the processes mentioned above.

In order to explain our results, let $Y = (Y_t)_{t \in [0,\infty)}$ be a general diffusion in natural scale with speed measure $m$. Assume for the rest of the introduction that the state space of $Y$ is equal to the whole real line.
In addition, let $(\xi_k)_{k \in \bbN}$ be an iid sequence of random variables, on a probability space with a measure $P$, satisfying $P(\xi_k = \pm 1) = \frac12$. 
Let $\ol h \in (0,\infty)$ and let $a_h\colon \R \to[0,\infty)$, $h\in (0,\ol h)$, be a family of functions.
Given some starting point $y\in \R$, for $h \in (0,\ol h)$
we denote by $X^h = (X^h_{kh})_{k \in \bbN_0}$ the Markov chain defined by
\begin{equation}\label{eq:def_X_intro}
X^h_0 = y
\quad\text{and}\quad
X^h_{(k+1)h} = X^h_{kh}+ a_h(X^h_{kh}) \xi_{k+1}, \quad \text{ for } k \in \bbN_0. 
\end{equation}
We extend $X^h$ to a continuous-time process by linearly interpolating two neighboring sequence elements; more precisely, we define 
\begin{equation}\label{eq:lin_interpo_intro}
X^h_t = X^h_{\lfloor t/h \rfloor h} + (t/h - \lfloor t/h \rfloor) (X^h_{(\lfloor t/h \rfloor +1)h} - X^h_{\lfloor t/h \rfloor h}), \qquad t\in[0,\infty). 
\end{equation}
In what follows, the term ``Markov chain''
is used also for the continuous-time process $X^h$
because it is always clear from the context
whether we speak about
the discrete-time Markov chain~\eqref{eq:def_X_intro}
or about the linearly-interpolated
continuous-time process $X^h$ of~\eqref{eq:lin_interpo_intro}.
The functional limit theorem in \cite{aku2018cointossing} says that if for all compact sets $K\subset \R$ the functions $(a_h)_{h \in (0,\ol h)}$ satisfy 
\begin{equation}\label{eq:charac_sf_intro}
\sup_{y \in K}\left|
\frac{1}{2}\int_{(y-a_h(y),y+a_h(y))} (a_h(y)-|u-y|)\,m(du)
-h
\right|\in o(h),
\end{equation}
then the associated Markov chains
$(X^h)_{h\in (0,\ol h)}$ converge in distribution to $Y$, as $h\to 0$.
To motivate the expression on the left-hand side of~\eqref{eq:charac_sf_intro},
we remark that, for $y\in\bbR$ and $a>0$, it holds
\begin{equation}\label{eq:14042020a1}
E_y[H_{y-a,y+a}(Y)]=\frac12\int_{(y-a,y+a)}(a-|u-y|)\,m(du)
\end{equation}
(see Remark~1.2 in \cite{aku2018cointossing}),
where $H_{y-a,y+a}(Y)=\inf\{t\ge0:Y_t\notin(y-a,y+a)\}$
($\inf\emptyset:=\infty$),
and $E_y$ denotes the expectation under the measure $P_y$
under which the Markov process $Y$ starts in $y$:
$P_y(Y_0=y)=1$.
Moreover, it is discussed in \cite{aku2018cointossing}
that, for any speed measure $m$,
there exist families $(a_h)_{h\in(0,\ol h)}$
satisfying~\eqref{eq:charac_sf_intro},
i.e., any general diffusion can be approximated in law
in this way.
Since random samples of $X^h$, $h\in(0,\ol h)$,
can be efficiently generated on a computer, this result opens the door for numerical approximations of
the distribution of $Y$ via Monte Carlo simulations of $X^h$ with $h$ chosen small enough.

Notice that we use symmetric Bernoulli
random variables $\xi_k$, $k\in\bbN$, in~\eqref{eq:def_X_intro}
rather than say Gaussian ones.
There are several reasons for that.
Firstly, this is especially convenient
in the case when the state space of $Y$
is a bounded subinterval in $\bbR$
because, by a suitable choice
of the functions $a_h$, $h\in(0,\ol h)$,
we can guarantee that the approximations
$X^h$
never leave the state space of $Y$
(see the precise description of the setting
in Section~\ref{sec:scheme}).
Since normal distributions 
have unbounded support,
this would be impossible in the case
of Gaussian $\xi_k$
regardless of the choice of $a_h$.
Secondly, there are technical arguments
in \cite{aku2018cointossing}
that do not go through with Gaussian $\xi_k$,
i.e., the functional limit theorem
with such a general scope
(capable of approximating \emph{any}
general diffusion)
exists by now for symmetric Bernoulli $\xi_k$
but not for Gaussian ones.
Finally, on certain machines
(like field programmable gate arrays)
it is more efficient to use \emph{random bits}
(as opposed to \emph{random numbers}),
see \cite{BSWOHRKK} and \cite{GHMR:19}
for more detail.
The algorithms based on~\eqref{eq:def_X_intro}
require only to generate random bits
and hence can be very efficiently implemented
on such machines.
We, therefore, sometimes refer to the processes
$X^h=(X^h_t)_{t\in[0,\infty)}$, $h\in(0,\ol h)$,
as \emph{random bit approximations} of~$Y$
or as \emph{random bit Markov chains}.

In the present paper we address the question of how
\emph{fast} the Markov chains $(X^h)_{h \in (0,\ol h)}$ converge to $Y$. 
In order to obtain our results we impose a stronger assumption than~\eqref{eq:charac_sf_intro}. Essentially, we assume that there exists $\lambda \in (0, \infty)$ such that the functions $(a_h)_{h \in (0,\ol h)}$ satisfy 
\begin{equation}\label{cond1}
\sup_{y \in \mathbb{R}} \left|
\frac{1}{2}\int_{(y-a_h(y),y+a_h(y))} (a_h(y)-|u-y|)\,m(du)
-h
\right| \in \mathcal{O}(h^{1+\lambda}),
\end{equation}
which is an assumption on the approximation scheme,
and we discuss that there always exist approximation schemes
satisfying~\eqref{cond1}.
In addition, we assume that the Cauchy distribution has a bounded density with respect to the speed measure of $Y$, i.e., that there exist $k_1 \in (0,\infty)$ and $k_2 \in \{0,1\}$ such that 
\begin{align}\label{cond2}
m(dx)\ge \frac{2}{k_1(1+k_2x^2)}\,dx
\end{align}
(see Remark~\ref{rem:07032019a1} for why we include a binary variable $k_2$ in~\eqref{cond2}).
It is worth noting that \eqref{cond2} does not exclude sticky features
mentioned above,
as the latter are modeled via atoms in the speed measure.
Our first convergence result, Theorem~\ref{thm:wasserstein_term} below, states, for every finite time $T \in (0, \infty)$
and every $p\in[1,\infty)$, that if \eqref{cond1} and \eqref{cond2} are satisfied, then the $p$-th Wasserstein distance between the law of $X^h_T$ and the law of $Y_T$ converges to zero at least at the rate $\min(\frac14, \frac{\lambda}{2})$, i.e., the $p$-th Wasserstein distance between $X^h_T$ and $Y_T$ is bounded by a constant times $h^{\min(\frac14, \frac{\lambda}{2})}$.
Under the same assumptions the $p$-th Wasserstein distance between the laws of the paths of $Y$ and $X^h$, up to $T$, converge at least at the rate $\min(\frac14, \frac{\lambda}{2})-\varepsilon$, where $\varepsilon$ is a positive real arbitrarily close to zero
(see Theorem~\ref{thm:wasserstein_path} below).
We remark that the convergence results apply also to general diffusions whose state spaces are intervals, but not necessarily the whole real line.
In such cases Conditions \eqref{cond1} and~\eqref{cond2} have to be slightly modified (cf.\ Condition~(A$\lambda$) and Condition~(C) below).

To provide more details for Condition~\eqref{cond1},
we note that for all $y\in \R$, $h\in (0,\ol h)$ there always exists $\wh a_h(y)\in (0,\infty)$ such that 
\begin{equation}\label{eq:emcel}
\frac{1}{2}\int_{(y-\wh a_h(y),y+\wh a_h(y))} (\wh a_h(y)-|u-y|)\,m(du)
=h.
\end{equation}
In particular, the family $(\wh a_h)_{h\in (0,\ol h)}$
satisfies~\eqref{cond1} for all $\lambda\in (0,\infty)$ and, therefore, the path distribution of \emph{any} general diffusion
satisfying~\eqref{cond2} can be approximated by
random bit
Markov chains at any Wasserstein rate strictly smaller than $1/4$. In practice, the solutions $(\wh a_h)_{h\in (0,\ol h)}$ cannot be determined in closed form but need to be approximated numerically themselves.
Therefore, results for the scheme
$(\wh a_h)_{h\in(0,\ol h)}$ of~\eqref{eq:emcel} should be complemented by a
\emph{perturbation analysis} for~\eqref{eq:emcel}.
This analysis is carried out in the present article and leads to the relaxed version~\eqref{cond1} of~\eqref{eq:emcel}.
Theorem~\ref{thm:wasserstein_path} thus ensures that if
Equation~\eqref{eq:emcel} is only solved with a precision of order $\mathcal{O}(h^{3/2})$ then the associated Markov chains $(X^h)_{h\in (0,\ol h)}$ still converge at any rate
arbitrarily close to $1/4$
(and, for the convergence of time marginals, Theorem~\ref{thm:wasserstein_term}
provides the rate $1/4$).
A further reduction of the precision entails a smaller rate.

Concerning Condition~\eqref{cond2},
we note that this assumption is not required for the functional limit theorem in \cite{aku2018cointossing}.
As we measure the convergence speed here with respect to
Wasserstein distances of all orders $p\in [1,\infty)$, we necessarily have to ensure that $Y$ and $X^h$, $h \in (0,\ol h)$, admit finite moments of all orders. This is achieved under~\eqref{cond2} in Section~\ref{sec:fin_mom} below. In fact, \eqref{cond2} is nearly a minimal assumption
required for our results (see Section~\ref{sec:CC}).

The proofs of our results rely on embeddings of the
random bit
Markov chains into the diffusion $Y$ with sequences of stopping times. More precisely,
for every $h\in (0,\ol h)$ we define the sequence of stopping times $(\tau^h_k)_{k \in \mathbb{N}_0}$ by the formulas
$\tau^h_0=0$,
$\tau^h_{k+1}=\inf\{t\ge\tau^h_k:Y_t\notin(Y_{\tau^h_k}-a_h(Y_{\tau^h_k}),Y_{\tau^h_k}+a_h(Y_{\tau^h_k}))\}$,
$k\in\bbN_0$,
and observe that the sequence $(Y_{\tau^h_k})_{k \in \mathbb{N}_0}$ has the same law as the sequence $(X^h_{kh})_{k \in \mathbb{N}_0}$.
Condition~\eqref{cond1} allows to conclude that $\tau^h_{\lfloor T/h \rfloor}$ converges to $T$ with respect to every $L^p$-norm, with $p \in [1,\infty)$, at a rate of at least $\min\{\frac{1}{2},\lambda\}$.
At this point the latter statement can be informally explained
by~\eqref{eq:14042020a1}.
With Condition~\eqref{cond2} and a representation of $Y$ as a time-changed Brownian motion we can estimate the $L^p$-distance between $Y_T$ and $Y_{\tau^h_{\lfloor T/h \rfloor}}$ against the distance between  $\tau^h_{\lfloor T/h \rfloor}$ and $T$, and conclude that $Y_{\tau^h_{\lfloor T/h \rfloor}}$ converges to $Y_T$ in $L^p$ at a rate of at least $\min\{\frac{1}{4},\frac{\lambda}{2}\}$. We stress that our analysis is optimal. Indeed, we show by means of an example that the rate of $L^p$-convergence for $p\in [4,\infty)$
cannot be improved beyond $1/4$.
As the pair $(Y_T, Y_{\tau^h_{\lfloor T/h \rfloor}})$ constitutes  a coupling between the laws of $Y_T$ and $X^h_{h\lfloor T/h \rfloor}$,
we arrive at
an estimate for the convergence rate with respect to the $p$-th Wasserstein distance.
A similar procedure can be used for proving the convergence rate for the path distributions.

\bigskip
The results in the present article generalize the results of our article \cite{aku-jmaa} in several 
directions. In \cite{aku-jmaa} we do not allow for {\itshape all} general diffusions, but consider only diffusions that solve an SDE of the form 
\begin{equation}\label{eq:04102017a1}
dY_t=\eta(Y_t)\,dW_t.
\end{equation}
It is known that \eqref{eq:04102017a1} has a unique in law weak solution
under the Engelbert-Schmidt condition that $\eta$
is a non-vanishing (possibly irregular) Borel function
such that $1/\eta^2$ is locally integrable
(see \cite{ES1985} or Theorem~5.5.7 in \cite{KS}). This is a special case of the setting considered in the present article with $m(dx)=\frac{2}{\eta^2(x)}dx$. The article \cite{aku-jmaa} does not provide a pertubation analysis, i.e., it
only analyzes the situation, where Equation~\eqref{eq:emcel} can be solved {\itshape exactly}
 and not only with a precision of the order $\mathcal{O}(h^{1+\lambda})$.
Moreover, \cite{aku-jmaa} considers the convergence
of the embedded Markov chains
with respect to the $L^2$-norm only
and for the time marginals only.
This neither allows to draw conclusions
about the convergence of the time marginals
in the Wasserstein distance of order $p>2$
nor to say anything about the convergence speed
for the path distributions.
It is shown in \cite{aku-jmaa} that the scheme has the rate $1/4$
for approximating the distribution of $Y$ at single points in time
under the assumption of global boundedness of $|\eta|$ and $1/|\eta|$.
In the present article we replace the boundedness assumption
by the weaker assumption~\eqref{cond2}.
In the context of SDEs, \eqref{cond2} imposes a linear growth condition
on $\eta$ (but no regularity condition).
Note, however, that \eqref{cond2} allows to go beyond SDEs,
e.g., it includes diffusions with sticky features.



As noted in \cite{AJKH}, the Wasserstein distance is an appropriate measure for the distance between the path distribution of a diffusion and its Markov chain approximation.
In \cite{AJKH} every $p$-th
Wasserstein distance between the path distribution of a Lipschitz continuous and uniformly elliptic SDE and its Euler approximation is shown to converge to zero with rate
$\frac23-\eps$ for every positive real~$\eps$.
As discussed above,
the approximations of our algorithm also converge with respect to every $p$-th Wasserstein distance;
we prove convergence of our algorithm in Wasserstein distances for a larger class of diffusions in natural scale but at a slower rate.

For classical results on the approximation of SDEs
with Lipschitz coefficients via the Euler scheme
we refer to the books
\cite{kloeden1992numerical}
and \cite{Pages:18}
and to references therein.
The Euler scheme is known to converge also for some SDEs with non-Lipschitz coefficients, and results on the convergence speed of the weak and strong approximation error exist; see, e.g.,
\cite{KHLY:JCAM2017}
and
\cite{NgoTaguchi:SPL2017}
and the references therein
for results on the weak and strong convergence
for the Euler-type approximations
of SDEs with discontinuous coefficients
(note, however, that the Euler scheme
may fail to converge in the numerically weak sense
even in the case of continuous coefficients,
see \cite{HJK}).
\cite{KonMen:17} and \cite{Frikha:18}
establish weak convergence rate
of certain Euler-type schemes for diffusions
with H\"older coefficients,
and the rate is a half of the H\"older exponent.
In contrast to our scheme,
however, the Euler scheme is defined only for SDEs
(not for \emph{all} general diffusions)
and may fail to converge even in the stochastically weak sense
when the SDE coefficients are irregular
(see Section~5.4 in \cite{aku-jmaa}).

The random bit Markov chains can be embedded
into the general diffusion $Y$ with a sequence of stopping times,
and hence the Markov chains can be interpreted
as exact appearances of $Y$ along a stochastic time grid.
There are a few related schemes in the literature
employing random time grids to approximate solutions
of one-dimensional SDEs.
In \cite{EL} the authors
construct Bernoulli random walks on finite grids in the space,
create random walks on certain associated random time grids,
and compute the rate with which these random walks
converge in $L^1$.
A similar scheme is introduced in \cite{milstein2015uniform}
for approximating the CIR process.
In contrast to \cite{EL}
the scheme of \cite{milstein2015uniform}
is exact along a certain sequence of stopping times.
A further scheme,
exact on a time grid with exponentially distributed steps,
and applicable to SDEs with discontinuous coefficients
is suggested in \cite{LLP2019}.
The results in
\cite{EL}, \cite{LLP2019} and \cite{milstein2015uniform}
apply to SDEs with drift, whereas the present article
puts a focus on arbitrary general diffusions in natural scale
and hence has a different scope.

\bigskip
To summarize the discussion in the introduction,
we can say that in the literature there are
plenty of results that study approximations
of solutions to SDEs.
There are, however, many general diffusions
that cannot be written as solutions to SDEs.
This is, in particular, the case for general diffusions
with sticky features and, in such a situation,
reference \cite{aku2018cointossing}
as well as the present paper come into play.
While reference \cite{aku2018cointossing}
deals with approximations in the sense
of weak convergence,
in the present paper we aim
at approximating general diffusions
in a stronger sense\footnote{Notice
that the convergence in every $p$-th
Wasserstein distance
implies the convergence of expectations
for all continuous path functionals
with polynomial growth,
whereas the weak convergence
is the convergence of expectations
for all bounded continuous path functionals.}
and prove the rate of $1/4$.
Essentially, the price for obtaining these stronger results
is Condition~\eqref{cond2}, which is, as discussed,
not required in \cite{aku2018cointossing}
but is nearly a minimal assumption
for the results of the present paper to hold.
Also, the techniques used in the proofs of this paper
differ from those in \cite{aku2018cointossing}:
while in \cite{aku2018cointossing}
we proceed through various moment estimates
for the embedding stopping times,
in the present paper we explicitly work with
the random time change involved
in the construction of a general diffusion.
Finally, it is worth noting that the rate of $1/4$
cannot be considered as too slow,
as in \emph{almost all}\footnote{The exceptions
are \cite{aku-jmaa} and \cite{EL}.
We recall that the paper \cite{EL}
discusses different results from ours
for an algorithm that has a different scope.
As for \cite{aku-jmaa}, we refer to the paragraph
containing~\eqref{eq:04102017a1},
where we explain in detail that the present paper
substantially generalizes \cite{aku-jmaa}
in many directions.} references above
involving approximation methods for SDEs
with some rate, the rate degenerates
as the diffusion coefficient loses regularity.
On the contrary, in our setting the diffusion coefficient
of an SDE is allowed to be arbitrarily irregular
(just a Borel function);
the algorithm even allows to consider
general diffusions with sticky points or the like,
which cannot be written as solutions to SDEs;
our scheme adapts in the way that preserves
the rate of $1/4$ in any case.

\bigskip
The paper is organized as follows.
In Section~\ref{sec:scheme}
we formally describe the processes we approximate,
the approximation schemes used for this aim
and formulate and discuss our main results.
Section~\ref{sec:prop_appr}
studies some useful properties of the approximation schemes.
Section~\ref{sec:fin_mom} establishes moment bounds
for the general diffusion and for the approximations.
In Section~\ref{sec:embedding}
we describe the embedding of the approximating
random bit Markov chains into the general diffusion
and present (optimal) $L^p$-convergence rates
for the involved stopping times.
After the preparations done in
Sections \ref{sec:prop_appr}, \ref{sec:fin_mom} and~\ref{sec:embedding},
we prove our main results in
Sections \ref{sec:rate} and~\ref{sec:rate_path}.
Finally, in Section~\ref{sec:CC}
we present examples showing
that Condition~\eqref{cond2}
is essential for our results.

\section{Approximation schemes and main results}\label{sec:scheme}
Let $(\Omega, \cF, (\cF_t)_{t \ge 0}, (P_y)_{y \in I}, (Y_t)_{t \ge 0})$ be a one-dimensional continuous strong Markov process in the sense of Section~VII.3 in \cite{RY}. We refer to this class of processes as {\itshape general diffusions} in the sequel. We assume that the state space is an open, half-open or closed interval $I \subseteq \R$. We denote by $I^\circ=(l,r)$ the interior of $I$, where $-\infty\leq l<r\leq \infty$, and we set $\ol I=[l,r]$.
Recall that by the definition we have $P_y[Y_0=y]=1$ for all $y\in I$. 
We further assume that $Y$ is regular. This means that for every $y\in I^\circ$ and $x\in I$ we have that $P_y[H_x(Y)<\infty]>0$, where $H_x(Y)=\inf\{t\geq 0: Y_t=x \}$
(with the usual convention $\inf\emptyset=\infty$).
If there is no ambiguity,
we simply write $H_x$ in place of $H_x(Y)$. Moreover, for $a<b$ in $\ol I$ we denote by $H_{a,b}=H_{a,b}(Y)$
the first exit time of $Y$ from $(a,b)$,
i.e.\ $H_{a,b} = H_a\wedge H_b$.
We suppose that the diffusion $Y$ is in natural scale. If $Y$ is not in natural scale, then there exists a strictly increasing continuous function $s:I \to \R$, the so-called scale function, such that $s(Y_t)$, $t\geq 0$, is in natural scale. 
Let $m$ be the speed measure of the Markov process $Y$
(see VII.3.7 and~VII.3.10 in \cite{RY}).
Recall that for all $a<b$ in $I^\circ$ we have
\begin{equation}\label{eq:06072018a1}
0<m([a,b])<\infty.
\end{equation}
Finally,
we always assume that if a boundary point is accessible, then it is absorbing. 
We refer to Remark~\ref{rem:reflection} on how to drop this assumption and to also allow for reflecting boundaries.

Let $\ol h \in (0,1)$ and suppose that for every $h \in (0, \ol h)$ we are given a measurable function $a_{h}\colon \ol I \to [0,\infty)$ such that $a_h(l)=a_h(r)=0$ and for all $y\in I^\circ$ we have $y\pm a_h(y)\in I$. We refer to each function $a_h$ as a \emph{scale factor}.
We next construct a sequence of Markov chains associated to the family of scale factors $(a_h)_{h\in (0, \ol h)}$.
To this end fix a starting point $y \in I^\circ$ of $Y$.
Let $(\xi_k)_{k \in \bbN}$ be an iid sequence of random variables,
on a probability space with a measure $P$,
satisfying $P(\xi_k = \pm 1) = \frac12$. 
We denote by $(X^h_{kh})_{k \in \bbN_0}$ the Markov chain defined by
\begin{equation}\label{eq:def_X}
X^h_0 = y
\quad\text{and}\quad
X^h_{(k+1)h} = X^h_{kh}+ a_h(X^h_{kh}) \xi_{k+1}, \quad \text{ for } k \in \bbN_0. 
\end{equation}
We extend $(X^h_{kh})_{k \in \bbN_0}$ to a continuous-time process
$X^h=(X^h_t)_{t\in[0,\infty)}$
by linear interpolation, i.e., for all $t\in[0,\infty)$, we set
\begin{equation}\label{eq:13112017a1}
X^h_t = X^h_{\lfloor t/h \rfloor h} + (t/h - \lfloor t/h \rfloor) (X^h_{(\lfloor t/h \rfloor +1)h} - X^h_{\lfloor t/h \rfloor h}). 
\end{equation}
The approximation schemes we consider in this paper
are the families of processes $(X^h)_{h\in(0,\ol h)}$
(which are encoded by the families $(a_h)_{h\in(0,\ol h)}$
of the scale factors).
To highlight the dependence of $X^h$ on the starting point $y\in I^\circ$ we also sometimes write~$X^{h,y}$.

To formulate our main results we need the additional assumption that
the Cauchy distribution has a bounded density
with respect to the speed measure.
More precisely, we suppose
that the following condition is satisfied.

\medskip

\textbf{Condition~(C)} 
There exist constants $k_1\in (0,\infty)$, $k_2\in\{0,1\}$ such that on $I^\circ$  we have
$m(dx)\ge \frac{2}{k_1(1+k_2x^2)}\,dx$, i.e., for all $A \in \mathcal B(I^\circ)$ we have $m(A)\ge \int_A \frac{2}{k_1(1+k_2x^2)}\,dx$.

\begin{remark}\label{rem:07032019a1}
In the formulation of Condition~(C) we allow for two degrees of freedom $k_1\in (0,\infty)$, $k_2\in\{0,1\}$. This might seem superfluous.
Indeed, if Condition~(C) is satisfied with $k_2=0$, then it also holds with $k_2=1$. 
In the results that follow, however, we get sharper bounds if Condition~(C) is satisfied with $k_2=0$. In particular, the constants appearing in Theorem~\ref{thm:wasserstein_term} and Theorem~\ref{thm:wasserstein_path} do not depend on the initial value $y$ of the diffusion $Y$ in the case $k_2=0$.
\end{remark}

\begin{ex}[Driftless SDE with possibly irregular diffusion coefficient]
A particular case of our setting is the case,
where $Y$ is a solution to the driftless SDE
\begin{equation}\label{eq:27092018a1}
dY_t=\eta(Y_t)\,dW_t,
\end{equation}
where $\eta\colon I^\circ\to\bbR$ is a Borel function
satisfying the Engelbert-Schmidt conditions
\begin{gather}
\eta(x)\ne0\;\;\forall x\in I^\circ,
\label{eq:27092018a2}\\[1mm]
\eta^{-2}\in L^1_{\loc}(I^\circ)
\label{eq:27092018a3}
\end{gather}
($L^1_{\loc}(I^\circ)$ denotes the set of Borel functions
locally integrable on~$I^\circ$).
Under \eqref{eq:27092018a2}--\eqref{eq:27092018a3}
SDE~\eqref{eq:27092018a1}
has a unique in law weak solution
(see \cite{ES1985} or Theorem~5.5.7 in \cite{KS}). 
In this case the speed measure of $Y$
on $I^\circ$ is given by the formula
\begin{equation}\label{eq:speed_measure_sde}
m(dx)=\frac 2{\eta^2(x)}\,dx.
\end{equation}
We refer to Section 2 in \cite{aku2018cointossing} 
for further details on the application of  the approximation scheme to SDEs.
If $m$ is given by \eqref{eq:speed_measure_sde}, Condition~(C) means that there
exist constants $k_1\in (0,\infty)$, $k_2\in\{0,1\}$ such that for all $x\in I^\circ$
it holds that $\eta(x)\le \sqrt{k_1(1+k_2x^2)}$, i.e., the diffusion coefficient $\eta$ is locally bounded and whenever the state space is unbounded $\eta$ is of at most  linear growth.
If $k_2=0$, then under Condition~(C) the diffusion $Y$ does not move faster than a Brownian motion scaled by~$\sqrt{k_1}$.
\end{ex}

To formulate our second main assumption we need to introduce some notation.
We introduce an auxiliary subset of $I^\circ$.
To this end, if $l> -\infty$, we define, for all $h\in(0,\ol h)$,
\begin{equation}\label{eq:22022019b1}
l_h = l+ \inf\left\{ a \in \left(0,\frac{r-l}2\right]:
a<\infty
\;\;\text{and}\;\;
\frac12\int_{(l, l+2a)} (a - |u-(l+a)|) m(du) \ge h \right\}, 
\end{equation}
where we use the convention
$\inf \emptyset = \frac{r-l}2$.
If $l = -\infty$, we set $l_h = -\infty$. 
Similarly, if $r < \infty$, then we define, for all $h\in(0,\ol h)$,
\begin{equation}\label{eq:22022019b2}
r_h = r- \inf\left\{ a \in \left(0,\frac{r-l}2\right]:
a<\infty
\;\;\text{and}\;\;
\frac12\int_{(r-2a, r)} (a - |u-(r-a)|) m(du) \ge h \right\}
\end{equation}
with the same convention $\inf \emptyset = \frac{r-l}2$.
If $r = \infty$, we set $r_h = \infty$.
In any case, it holds $l_h\le r_h$ and, moreover, $l_h\le\frac{l+r}2\le r_h$ whenever $l$ and $r$ are finite.
The auxiliary subset is defined by
\begin{equation}\label{eq:22022019a2}
I_h = (l_h,r_h)\cup\left\{y\in I^\circ \colon y\pm a_h(y)\in I^\circ\right\}.
\end{equation}
With the help of Feller's test for explosions (see, e.g., Theorem~23.12 in \cite{Kallenberg2002} or Lemma~2.1 in \cite{AKKK17} or Theorem~3.3 in \cite{AKKU2018}) one can show that $l$ is inaccessible if and only if $l_h = l$ for all $h\in (0,\ol h)$ (for details, see Remark~4.2 in \cite{aku2018cointossing}). Similary, $r$ is inaccessible if and only if $r_h = r$ for all $h\in (0,\ol h)$. In particular, if both $l$ and $r$ are inaccessible,
then it holds for all $h\in (0,\ol h)$ that $I_h=I=I^\circ$. If $l$ or $r$ are accessible it holds that $l_h\searrow l$ or $r_h\nearrow r$, respectively, as $h\searrow 0$.

The second assumption we need for our main results is the following condition on the scale factors.

\textbf{Condition~(A$\lambda$)} There exist $\lambda\in (0,\infty)$, $K\in [0,\infty)$
 such that for all $h\in (0,\ol h)$ it holds
 \begin{equation}\label{eq:condalambda}
\sup_{y\in I_h } 
\left|
\frac{1}{2}\int_{(y-a_h(y),y+a_h(y))} (a_h(y)-|u-y|)\,m(du)
-h
\right|
\le K h^{1+\lambda}
\end{equation}
For every $\lambda\in (0,\infty)$ Condition~(A$\lambda$) implies 
Condition~(A) in \cite{aku2018cointossing}. We refer to the introduction in 
\cite{aku2018cointossing} for an
interpretation of the left-hand side of \eqref{eq:condalambda} as an upper bound for the one-step temporal error of the approximation scheme.

Throughout the paper we make a convention that
if Condition~(A$\lambda$) is satisfied for some $\lambda\in (0,\infty)$, then $\ol h$ is chosen small enough such that
for all $h\in (0,\ol h)$ it holds
\begin{equation}\label{eq:gamma_cond}
\sup_{y\in I_h } 
\left|
\frac{1}{2}\int_{(y-a_h(y),y+a_h(y))} (a_h(y)-|u-y|)\,m(du)
-h
\right|\le \gamma h
\end{equation}
with some $\gamma \in [0,1)$.
In fact, in many auxiliary results that follow
(see Sections \ref{sec:prop_appr} and~\ref{sec:fin_mom})
we assume only the weaker Condition~\eqref{eq:gamma_cond}.
The stronger Condition~(A$\lambda$)
is generally needed to get rates of convergence.

Before formulating our main results
we present an example of an approximation scheme,
called EMCEL in \cite{aku2018cointossing},
which is well-defined for every general diffusion $Y$
and satisfies Condition~(A$\lambda$)
for all $\lambda\in(0,\infty)$.

\begin{ex}[EMCEL approximations]\label{ex:25022019a1}
Let $h\in(0,\ol h)$.
The EMCEL$(h)$ scale factor $\wh a_h$ is defined by
$\wh a_h(l)=\wh a_h(r)=0$
and, for all $y\in I^\circ$,
\begin{equation}\label{sf absorbing case}
\wh a_h(y) = \sup\left\{a \ge 0: y\pm a \in I \text{ and } \frac{1}{2}\int_{(y-a,y+a)} (a-|z-y|)\,m(dz) \le h\right\}.
\end{equation}
The associated process
defined in \eqref{eq:def_X}--\eqref{eq:13112017a1}
is denoted by $\wh X^h$ and referred to as
Embeddable Markov Chain with Expected time Lag~$h$
(we write shortly $\wh X^h\in\text{EMCEL}(h)$).\footnote{The
whole family $(\wh X^h)_{h\in (0,\ol h)}$ is referred to as the
\emph{EMCEL approximation scheme.}
Alternatively, we simply say
\emph{EMCEL approximations.}}
Expression~\eqref{sf absorbing case}
can be alternatively described as follows.
For $y\in(l_h,r_h)$, $\wh a_h(y)$ is a unique positive root
of the equation (in~$a$)
\begin{equation}\label{eq:22022019a3}
\frac{1}{2}\int_{(y-a,y+a)} (a-|z-y|)\,m(dz) = h,
\end{equation}
while, for $y\in(l,l_h]$ (resp., $y\in[r_h,r)$),
$\wh a_h(y)$ is chosen to satisfy
\begin{equation}\label{eq:22022019a4}
y-\wh a_h(y)=l\qquad\text{(resp., }y+\wh a_h(y)=r).
\end{equation}
It follows from~\eqref{eq:22022019a4} that
the set $I_h$ of~\eqref{eq:22022019a2}
corresponding to the EMCEL$(h)$ scale factor $\wh a_h$
(and, naturally, denoted by $\wh I_h$) is simply
$$
\wh I_h=(l_h,r_h).
$$
This yields that the left-hand side of~\eqref{eq:condalambda}
vanishes for the EMCEL approximations and,
therefore, for such a scheme Condition~(A$\lambda$)
is satisfied for all $\lambda\in(0,\infty)$.
\end{ex}

We first formulate our main result concerning 
the Wasserstein approximation rate for time marginals of $Y$.
To this end we briefly recall the definitions of Wasserstein distances.
For $p\in [1,\infty)$ let $\mathcal M_p(\R)$ denote the set of all probability measures on $(\R, \mathcal B(\R))$ with finite $p$-th moment, i.e., probability measures $\mu$ satisfying $\int |x|^p \mu(dx)<\infty$.
The $p$-th Wasserstein distance between two probability measures $\mu, \nu \in \mathcal M_p(\R)$ is defined by 
$$
\mathcal W_p(\mu,\nu)=\inf \|\xi -\zeta\|_{L^p},
$$
where the infimum is taken over all
\emph{couplings} between $\mu$ and $\nu$,
that is, over all
random vectors $(\xi,\zeta)$ with marginals $\mu$ and $\nu$
(i.e., $\xi\sim \mu$ and $\zeta\sim \nu$).

\begin{theo}\label{thm:wasserstein_term}
Suppose that Condition~(C) is satisfied and that there exists $\lambda \in (0,\infty)$ such that Condition~(A$\lambda$)
holds. Let $p\in [1,\infty)$ and $T\in (0,\infty)$.
Then there exists a constant $C(p,T) \in [0,\infty)$ such that for all $y\in I^\circ$ and $h\in (0,\ol h)$ it holds
$Y_T\in L^p(P_y)$,
$X^{h,y}_{h\lfloor T/h \rfloor}\in L^p(P)$,
$X^{h,y}_T\in L^p(P)$,
\begin{equation}\label{eq:24022019c1}
\mathcal W_p(P\circ (X^{h,y}_{h\lfloor T/h \rfloor})^{-1}, P_y \circ (Y_T)^{-1})  \le C(p,T)
(1+k_2 |y|)
h^{\min\{\frac{1}{4},\frac{\lambda}{2}\}}
\end{equation}
and
\begin{equation}\label{eq:24022019c2}
\mathcal W_p(P\circ (X^{h,y}_T)^{-1}, P_y \circ (Y_T)^{-1})  \le C(p,T)
(1+k_2 |y|)
h^{\min\{\frac{1}{4},\frac{\lambda}{2}\}}.
\end{equation}
\end{theo}

\begin{remark}
In the case where $Y$ does not move faster than a scaled Brownian motion (i.e., $k_2=0$), the convergence in Theorem~\ref{thm:wasserstein_term} is uniform in the starting point $y\in I^\circ$ (the same applies to Theorem~\ref{thm:wasserstein_path} below).
\end{remark}

The proof of Theorem~\ref{thm:wasserstein_term} is provided in Section~\ref{sec:rate}.
Next, we describe  the rate at which
the law of $X^{h,y}$ converges to the one of $Y$ with respect to
the $p$-th Wasserstein distance
on the path space $C([0,T],I)$
endowed with the sup norm.
In what follows, we use the notation
$$
\|x\|_C=\sup_{t\in[0,T]}|x(t)|
$$
for the sup norm on $C([0,T],I)$.
Let $\mathcal M_p(C([0,T],I))$ denote the set of all probability measures $\mu$ on $C([0,T],I)$
(equipped with the Borel $\sigma$-field)
with finite $p$-th moment, that is,
$\int \|x\|_C^p\,\mu(dx)<\infty$.
The $p$-th Wasserstein distance between
$\mu,\nu\in\mathcal M_p(C([0,T],I))$ is defined by 
$$
\mathcal W_p(\mu,\nu)=\inf \left \|\sup_{t\in [0,T]}|\xi_t -\zeta_t| \right \|_{L^p},
$$
where the infimum is taken over all couplings between $\mu$ and $\nu$, i.e., over all random elements $(\xi,\zeta)$ taking values in $C([0,T],I^2)$ with marginals $\mu$ and $\nu$
(that is, $\xi\sim \mu$ and $\zeta\sim \nu$).

\begin{theo}\label{thm:wasserstein_path}
Suppose that Condition~(C) is satisfied and that there exists $\lambda \in (0,\infty)$ such that Condition
(A$\lambda$) holds. Let $p\in [1,\infty)$, $T\in (0,\infty)$ and $\varepsilon \in (0,\min\{\frac{1}{4},\frac{\lambda}{2}\})$.
Then there exists a constant $C(p,\varepsilon, T) \in [0,\infty)$ such that for all $y\in I^\circ$ and $h\in (0,\ol h)$ it holds that
\begin{equation}\label{eq:24022019d1}
P_y \circ (Y)^{-1}
\text{ and }
P\circ (X^{h,y})^{-1}
\text{ are in }
\cM_p(C([0,T],I))
\end{equation}
and
$$
\mathcal W_p(P\circ (X^{h,y})^{-1}, P_y \circ (Y)^{-1}) \le C(p,\varepsilon,T) (1+k_2|y|) h^{\min\{\frac{1}{4},\frac{\lambda}{2}\}-\varepsilon}.
$$
\end{theo}

The proof of Theorem~\ref{thm:wasserstein_path} is provided in Section~\ref{sec:rate_path}.

\medskip
In practice one is often interested in approximating
$E_y[f(Y_T)]$ or $E_y[F(Y_.)]$
for functions $f\colon I\to\bbR$
and path functionals $F\colon C([0,T],I)\to\bbR$.
The next result presents the respective rates.

\begin{theo}\label{cor:26022019a1}
Suppose that Condition~(C) is satisfied and that there exists $\lambda\in (0,\infty)$ such that Condition~(A$\lambda$) holds.
Let $T\in (0,\infty)$ and let $F \colon C([0,T],I) \to \R$ be a locally Lipschitz continuous path functional
with polynomially growing Lipschitz constant,
i.e., there exist $L,\alpha\in [0,\infty)$ such that
for all $x_1,x_2 \in C([0,T],I)$ it holds
\begin{equation}\label{eq:LipConstPol}
|F(x_1)-F(x_2)|\le L\left\{1+(\|x_1\|_C\vee\|x_2\|_C)^\alpha\right\}\|x_1-x_2\|_C.
\end{equation}
Then for every $\eps\in(0,\min\{\frac{1}{4},\frac\lambda 2\})$
there exist a constant $C(\alpha,\eps,T)\in[0,\infty)$ such that
for all $h\in (0,\ol h)$ and $y\in I^\circ$ it holds
\begin{equation}\label{eq:rate_path_locLip}
\left|E\left[F(X^{h,y}_{t};\,t\in[0,T])\right]
-E_y\left[F(Y_t;\,t\in[0,T])\right]\right|\
\le L C(\alpha,\eps,T) (1+|y|^\alpha) (1+k_2 |y|)
h^{\min\{\frac{1}{4},\frac{\lambda}{2}\}-\eps},
\end{equation}
where we use the convention $0^0:=1$
(such an expression can appear in the case $\alpha=0$).
\end{theo}

\begin{remark}
Note that both expectations in~\eqref{eq:rate_path_locLip}
are well-defined due to~\eqref{eq:LipConstPol}
(which implies polynomial growth of $F$)
and existence of all polynomial moments
(see~\eqref{eq:24022019d1}). Furthermore, 
we stress that
the special case
of a globally Lipschitz path functional $F$
is explicitly included in~\eqref{eq:LipConstPol}
by the possibility $\alpha=0$
and results in a better dependence on $y$
on the right-hand side of~\eqref{eq:rate_path_locLip}.
\end{remark}

Theorem~\ref{cor:26022019a1} is essentially deduced from
Theorem~\ref{thm:wasserstein_path},
but we also need several other technical results.
Therefore, we postpone the proof till Section~\ref{sec:rate_path}.

In the same vein, we get the rate $\min\{\frac14,\frac\lambda2\}$
when approximating $E_y[f(Y_T)]$
for a locally Lipschitz continuous
function $f\colon I\to\bbR$
with polynomially growing Lipschitz constant.
A precise formulation is similar to that of
Theorem~\ref{cor:26022019a1}.
In the proof one needs to apply Theorem~\ref{thm:wasserstein_term}
in place of Theorem~\ref{thm:wasserstein_path}.

\medskip
As we see in all previous statements,
the number $\lambda\in(0,\infty)$ in Condition~(A$\lambda$)
is essential for the convergence rate.
It is, therefore, reasonable to expect that the convergence in
every $p$-th Wasserstein distance (not necessarily with a rate)
holds only under Condition~(C) and the condition that the left-hand side
in~\eqref{eq:condalambda} is $o(h)$, as $h\to0$.
This is indeed true but does not follow from the previous results.
We formulate this as a separate theorem,
which is also proved in Section~\ref{sec:rate_path}.

\begin{theo}\label{th:15042020a1}
Assume that Condition~(C) holds and that
\begin{equation}\label{eq:15042020a1}
\sup_{y\in I_h } 
\left|
\frac{1}{2}\int_{(y-a_h(y),y+a_h(y))} (a_h(y)-|u-y|)\,m(du)
-h
\right|
\in o(h),\quad h\to0.
\end{equation}
Let $T\in (0,\infty)$, $y\in I^\circ$ and $p\in[1,\infty)$.
Then we have~\eqref{eq:24022019d1} and, moreover,
\begin{equation}\label{eq:15042020a2}
\mathcal W_p(P\circ (X^{h,y})^{-1}, P_y \circ (Y)^{-1})\to0,\quad h\to0,
\end{equation}
or, equivalently, for all continuous path functionals
$F\colon C([0,T],I)\to\bbR$ satisfying
$$
|F(x)|\le L(1+\|x\|_C^p),\quad x\in C([0,T],I)
$$
($L\in[0,\infty)$ is an arbitrary constant), it holds
\begin{equation}\label{eq:15042020a3}
E\left[F(X^{h,y}_{t};\,t\in[0,T])\right]
\to
E_y\left[F(Y_t;\,t\in[0,T])\right],
\quad h\to 0.
\end{equation}
\end{theo}

\begin{remark}[Applicability of the main results]
(i) [Theoretical view]
The assumptions in our main results,
Theorems \ref{thm:wasserstein_term}
and~\ref{thm:wasserstein_path},
are Condition~(C) and Condition~(A$\lambda$).
Condition~(C) is an assumption
on the general diffusion $Y$ we want to approximate,
and it cannot be dropped
(we present examples in Section~\ref{sec:CC}).
Condition~(A$\lambda$) is an assumption
on the approximation scheme.
The main results thus provide convergence rates
of the Markov chain approximation schemes
\eqref{eq:def_X}--\eqref{eq:13112017a1}
defined by scale factors
satisfying Condition~(A$\lambda$),
with some $\lambda\in(0,\infty)$,
to general diffusions satisfying Condition~(C).

Given a task to approximate some general diffusion $Y$
we usually have a choice of an approximation scheme,
and we always can choose a scheme
that gives the best rates in
Theorems \ref{thm:wasserstein_term}
and~\ref{thm:wasserstein_path}
(namely, rate $\frac14$ in Theorem~\ref{thm:wasserstein_term}
and rate $\frac14-$ in Theorem~\ref{thm:wasserstein_path})
whenever $Y$ satisfies Condition~(C).
Indeed, we can take the EMCEL approximation scheme (see Example~\ref{ex:25022019a1}),
which is well-defined for every general diffusion $Y$
and satisfies Condition~(A$\lambda$)
for all $\lambda\in(0,\infty)$.

\smallskip
(ii) [Towards a practical implementation]
In general, Equation~\eqref{eq:22022019a3}
defining the EMCEL scale factors
can rarely be solved in a closed form.
Therefore, in practice we often need to
use scale factors $(a_h)_{h\in(0,\ol h)}$
that are approximations of the EMCEL scale factors
$(\wh a_h)_{h\in(0,\ol h)}$,
and for this latter case the main results tell us
with which precision to solve~\eqref{eq:22022019a3}
in order still to achieve the best rates.
Namely, Condition~(A$\lambda$) with $\lambda=\frac12$
still guarantees rate $\frac14$ in Theorem~\ref{thm:wasserstein_term}
and rate $\frac14-$ in
Theorem~\ref{thm:wasserstein_path}.
As an informal summary, we get the following recipe.
Solving~\eqref{eq:22022019a3}
with a precision of order $\cO(h^{3/2})$
provides the best rates in the main results.
A further reduction of the precision
in solving~\eqref{eq:22022019a3} entails smaller rates.

Also recall that our results are presented for a general diffusion $Y$ in natural scale.
If $Y$ is not in natural scale, we should apply our results to $s(Y)$,
where $s$ is the scale function of $Y$.
For instance, if $Y$ is driven by an SDE with drift,
$s$ is represented by a certain integral that involves the SDE coefficients.
This integral can rarely be computed in a closed form,
but the procedure in this case is to write down
Equation~\eqref{eq:22022019a3}
for $s(Y)$ and solve it with the precision of order $\cO(h^{3/2})$
(in other words, there is no need to approximate $s$ in an intermediate step;
but approximating the solution to~\eqref{eq:22022019a3} is necessary
and can be, of course, a challenging problem).

On the other hand, our results should not be understood in the way that,
in all cases, we insist on solving~\eqref{eq:22022019a3}
with the precision of order $\cO(h^{3/2})$.
In ``nice'' situations (e.g., SDEs with ``sufficiently regular'' coefficients)
one may fall back on simpler schemes, e.g., the Euler one,
and check if its scale factors satisfy Condition~(A$\lambda$) with $\lambda=\frac12$.
We suggest to solve~\eqref{eq:22022019a3}
with the precision of order $\cO(h^{3/2})$ only in ``irregular'' cases,
which, however, can be an involved task.
To deal with this problem in general,
it is helpful to study properties of the EMCEL scheme,
which we do in \cite{aklu2020emcel}.
In particular, \cite{aklu2020emcel} contains a certain ODE
for the EMCEL scale factors $\wh a_h$, $h\in(0,\ol h)$,
which holds in full generality, i.e., for \emph{all} possible
speed measures~$m$.

Finally, it is worth mentioning that the EMCEL scheme is explicitly described
in a number of SDE examples in Section~5 of \cite{aku-jmaa}
and for sticky Brownian motion in Section~7 of \cite{aku2018cointossing}.
In addition, in Section~8 of \cite{aku2018cointossing}
a Brownian motion slowed down on the Cantor set
is successfully approximated in closed form
by means of approximate solutions to~\eqref{eq:22022019a3}
(the latter case is an example, where there is no chance to have an explicit solution to~\eqref{eq:22022019a3}, but we can still write
an explicit approximation scheme).
\end{remark}

\begin{remark}[Reflecting boundaries]\label{rem:reflection}
We described the approximation scheme
under the assumption that
if a boundary point is accessible, then it is absorbing.
The case of reflecting boundaries
can be reduced to this situation.
In more detail, we can include both instantaneously reflecting
and slowly reflecting (sticky) boundaries as follows.
If $Z$ is a general diffusion in natural scale with reflecting boundaries, then there exists a general diffusion $Y$ in natural scale
with inaccessible and/or absorbing
boundaries and a Lipschitz continuous function $f$ with Lipschitz constant $1$ such that $Z \stackrel{d}{=} f(Y)$ (see Section~6
of \cite{aku2018cointossing} for details; in the special case, where $[l,\infty)$ is the state space of $Z$ and $l\in\bbR$ a reflecting boundary,
we can find $Y$ being a certain general diffusion on $\R$
and take $f(y)=l+|y-l|$, $y\in \R$).
Let $X^h$, $h\in (0,\ol h)$,
be the processes of \eqref{eq:def_X}--\eqref{eq:13112017a1}
approximating $Y$.
To approximate $Z$ we define $Z^h=f(X^h)$, $h\in(0,\ol h)$.
As $f$ has Lipschitz constant~$1$,
it holds for all $h\in (0,\ol h)$ and $p\in [1,\infty)$ that 
$$
\mathcal W_p(P\circ (Z^{h})^{-1}, P_y \circ (Z)^{-1})\le \mathcal W_p(P\circ (X^{h})^{-1}, P_y \circ (Y)^{-1}),
$$
hence, the rates of
Theorems \ref{thm:wasserstein_term} and~\ref{thm:wasserstein_path}
are preserved also for the convergence
of $(Z^h)_{h\in (0,\ol h)}$ to $Z$
in every $p$-th Wasserstein distance.
\end{remark}

\section{Properties of the approximation schemes}\label{sec:prop_appr}
In this section we collect some useful properties
of approximation schemes
satisfying Condition~\eqref{eq:gamma_cond}.

\begin{lemma}\label{lem:sf_conv_zero}
Assume that \eqref{eq:gamma_cond} is satisfied.
Then for all $y\in I^\circ$ we have
\begin{equation}\label{eq:22022019a1}
\lim_{h\to0}a_h(y)=0.
\end{equation}
In particular, \eqref{eq:22022019a1} holds under Condition~(A$\lambda$)
with an arbitrary $\lambda\in(0,\infty)$.
\end{lemma}

\begin{proof}
We first notice that any $y\in I^\circ$ belongs to $I_h$
whenever $h\in(0,\ol h)$ is sufficiently small.
Hence, \eqref{eq:gamma_cond} implies that
for all $y\in I^\circ$
$$
\lim_{h\to0}
\int_{(y-a_h(y),y+a_h(y))}
(a_h(y)-|u-y|)\,m(du)
=0.
$$
Let $y\in I^\circ$ and assume that 
$\limsup_{h\to 0}a_h(y)>0$.
Then there exists $\eps>0$ and a sequence
$(h_n)_{n\in \N}\subset (0,\ol h)$
such that $\lim_{n\to \infty} h_n=0$ and $a_{h_n}(y)\ge \eps$ for all $n\in \N$.
Then it holds for all $n\in \N$ that
$$
\int_{(y-a_{h_n}(y),y+a_{h_n}(y))} (a_{h_n}(y)-|u-y|)\,m(du)
\ge 
\int_{(y-\eps,y+\eps)} (\eps-|u-y|)\,m(du)>0.
$$
The obtained contradiction completes the proof.
\end{proof}

We fix an arbitrary $y_0\in I^\circ$.
In what follows we make a convention that
if \eqref{eq:gamma_cond} is satisfied,
then $\ol h$ is chosen small enough to guarantee
\begin{equation}\label{eq:02032019a1}
\sup_{h\in(0,\ol h)}a_h(y_0)<\infty.
\end{equation}
Indeed, due to Lemma~\ref{lem:sf_conv_zero},
under~\eqref{eq:gamma_cond} we can always
achieve~\eqref{eq:02032019a1}
by reducing $\ol h$ if necessary.

\begin{propo}\label{prop:approx_comp}
Assume that \eqref{eq:gamma_cond} is satisfied.
Then it holds for all $h\in (0,\ol h)$ and $y_1,y_2 \in I_h$ with $y_1<y_2$ that
\begin{equation}\label{eq:approx_comp}
y_1-y_2-\frac{2\gamma a_h(y_2)}{1-\gamma}	\le a_h(y_2)-a_h(y_1)\le y_2-y_1 +
\frac{2\gamma a_h(y_1)}{1-\gamma}.
\end{equation}
Moreover, there exists $C\in [0,\infty)$ such that for all $h\in (0,\ol h)$ and $y \in I^\circ$ we have
\begin{equation}\label{eq:lin_growth_sf_woh}
a_h(y)\le C+|y|.
\end{equation}
In particular, both statements hold under Condition~(A$\lambda$)
with an arbitrary $\lambda\in(0,\infty)$.
\end{propo}

\begin{proof}
For the first part of the proof fix $h\in (0,\ol h)$ and $y_1,y_2 \in I_h$ with $y_1<y_2$.
We only prove the second inequality in \eqref{eq:approx_comp}. The first one follows from similar considerations. Assume that
$a_h(y_2)-a_h(y_1)\ge y_2-y_1$ (else there is nothing to show). Note that \eqref{eq:gamma_cond} ensures that $a_h(y_1)>0$. It holds that
\begin{equation}
\begin{split}
&\frac{1}{2}\int_{(y_1-a_h(y_1),y_1+a_h(y_1))} \left(a_h(y_1)-|u-y_1|\right)\,m(du)
\left(
1+\frac{a_h(y_2)-y_2-a_h(y_1)+y_1}{a_h(y_1)}
\right)\\
&=
\frac{1}{2}\int_{(y_1-a_h(y_1),y_1+a_h(y_1))} \left(a_h(y_1)-|u-y_1|\right)\,m(du)\\
&\qquad 
+\frac{a_h(y_2)-y_2-a_h(y_1)+y_1}{2a_h(y_1)}
\int_{(y_1-a_h(y_1),y_1+a_h(y_1))} \left(a_h(y_1)-|u-y_1|\right)\,m(du)
\\
&\le \frac{1}{2}\int_{(y_1-a_h(y_1),y_1+a_h(y_1))} \left(a_h(y_1)-|u-y_1|\right)\,m(du)\\
&\qquad 
+\frac{1}{2}(a_h(y_2)-y_2-a_h(y_1)+y_1)\, m((y_1-a_h(y_1),y_1+a_h(y_1)))\\
&=\frac{1}{2}\int_{(y_1-a_h(y_1),y_1+a_h(y_1))} \left(a_h(y_2)-|u-y_1|-y_2+y_1\right)\,m(du).
\end{split}
\end{equation}
Combining this with the fact that $\forall u\in\R \colon |u-y_2|\le |u-y_1|+|y_1-y_2|=|u-y_1|+y_2-y_1$ implies that
\begin{multline}\label{aux_comp_1}
\frac{1}{2}\int_{(y_1-a_h(y_1),y_1+a_h(y_1))} \left(a_h(y_1)-|u-y_1|\right)\,m(du)
\left(
1+\frac{a_h(y_2)-y_2-a_h(y_1)+y_1}{a_h(y_1)}
\right)\\
\le 
\frac{1}{2}
\int_{(y_1-a_h(y_1),y_1+a_h(y_1))} \left(a_h(y_2)-|u-y_2|\right)\,m(du).
\end{multline}
The assumption that $a_h(y_2)-a_h(y_1)\ge y_2-y_1$ ensures that 
$(y_1-a_h(y_1),y_1+a_h(y_1))\subseteq (y_2-a_h(y_2),y_2+a_h(y_2))$. Moreover it holds that $a_h(y_2)-|u-y_2|\ge 0$ for all $u\in (y_2-a_h(y_2),y_2+a_h(y_2))$. Combining this with \eqref{aux_comp_1} shows that
\begin{multline}
\frac{1}{2}\int_{(y_1-a_h(y_1),y_1+a_h(y_1))} \left(a_h(y_1)-|u-y_1|\right)\,m(du)
\left(
1+\frac{a_h(y_2)-y_2-a_h(y_1)+y_1}{a_h(y_1)}
\right)\\
\le 
\frac{1}{2}
\int_{(y_2-a_h(y_2),y_2+a_h(y_2))} \left(a_h(y_2)-|u-y_2|\right)\,m(du).
\end{multline}
It follows from \eqref{eq:gamma_cond} that
\begin{equation}
\frac{1}{2}
\int_{(y_2-a_h(y_2),y_2+a_h(y_2))} \left(a_h(y_2)-|u-y_2|\right)\,m(du)
\le (1+\gamma)h
\end{equation}
and that
\begin{equation}
\frac{1}{2}\int_{(y_1-a_h(y_1),y_1+a_h(y_1))} \left(a_h(y_1)-|u-y_1|\right)\,m(du)\ge (1-\gamma)h.
\end{equation}
Hence, it holds that
\begin{equation}
(1-\gamma)h\left(
1+\frac{a_h(y_2)-y_2-a_h(y_1)+y_1}{a_h(y_1)}
\right)\le  (1+\gamma)h
\end{equation}
This implies that
\begin{equation}
a_h(y_2)-a_h(y_1)\le y_2-y_1 +
\frac{2\gamma a_h(y_1)}{1-\gamma}.
\end{equation}
It remains to prove Inequality~\eqref{eq:lin_growth_sf_woh}.
Applying~\eqref{eq:02032019a1} together with
the second inequality in~\eqref{eq:approx_comp}
with $y_1=y_0$ and $y_2>y_0$
as well as the first inequality in~\eqref{eq:approx_comp}
with $y_2=y_0$ and $y_1<y_0$,
we infer that there exists $\wt C\in [0,\infty)$ such that,
for all pairs $(h,y)$ with $h\in (0,\ol h)$, $y\in I_h$,
it holds $a_h(y)\le \wt C+|y|$.
Furthermore, for all pairs $(h,y)$ with $h\in(0,\ol h)$,
$y\in I^\circ\setminus I_h$,
we have $y+a_h(y)=r$ or $y-a_h(y)=l$,
which means that $y\mapsto a_h(y)$ is affine with slope $\pm1$
and intercept $r$ or $-l$.
This completes the proof.
\end{proof}


Proposition~\ref{prop:approx_comp} provides the functional bound
$C+|y|$ (independent of~$h$)
for all functions $a_h$, $h\in(0,\ol h)$, satisfying~\eqref{eq:gamma_cond}.
For any fixed $y\in I^\circ$, we know from
Lemma~\ref{lem:sf_conv_zero}
that $a_h(y)\to0$ as $h\to0$ (again under~\eqref{eq:gamma_cond}).
A natural question is, therefore, to find a functional bound
for $a_h$ that depends on $h$ and vanishes as $h\to0$.
Under Condition~(C) in addition to~\eqref{eq:gamma_cond},
such a bound is suggested in the next result.

\begin{propo}\label{prop:cond_d_lin_growth}
Suppose that Condition~(C) is satisfied
and that \eqref{eq:gamma_cond} holds true.
Then there exists a constant $C\in (0,\infty)$ such that
for all $h\in (0,\ol h)$ and $y\in I^\circ$ it holds that
\begin{equation}\label{eq:lin_growth_sf}
a_h(y)\le C (1+k_2|y|)\sqrt{h}.
\end{equation}
In particular, \eqref{eq:lin_growth_sf} holds under Condition~(C)
and Condition~(A$\lambda$)
with an arbitrary $\lambda\in(0,\infty)$.
\end{propo}

\begin{proof}
By~\eqref{eq:gamma_cond}, it holds
for $h\in (0,\ol h)$ and $y\in I_h$ that
\begin{equation}\label{eq:lin_growth_sf_aux}
\frac{1}{2}\int_{(y-a_h(y),y+a_h(y))} (a_h(y)-|u-y|)\,m(du)
\le(\gamma+1)h.
\end{equation}
It follows from \eqref{eq:22022019b1},
\eqref{eq:22022019b2}
and~\eqref{eq:22022019a2}
that for $h\in (0,\ol h)$ and $y\in I^\circ\setminus I_h$
$$
\frac{1}{2}\int_{(y-a_h(y),y+a_h(y))} (a_h(y)-|u-y|)\,m(du)
\le h.
$$
In particular, \eqref{eq:lin_growth_sf_aux} holds, in fact,
for all $h\in (0,\ol h)$ and $y\in I^\circ$.
Now Condition~(C) implies for all $h\in (0,\ol h)$ and $y\in I^\circ$ 
\begin{equation}\label{eq:lin_growth_sf_aux2}
\begin{split}
(\gamma+1)h
&\ge \int_{y-a_h(y)}^{y+a_h(y)} \frac{a_h(y)-|u-y|}{k_1(1+k_2u^2)}\,du
=a^2_h(y) \int_{-1}^1\frac{1-|z|}{k_1(1+k_2(y+za_h(y))^2)}dz\\
&\ge \frac{a^2_h(y)}{k_1(1+k_2\max_{z\in [-1,1]}(y+z a_h(y))^2)} \int_{-1}^1( 1-|z|)dz\\
&=\frac{a^2_h(y)}{k_1(1+k_2(|y|+ a_h(y))^2)}.
\end{split}
\end{equation}
Combining this with Proposition~\ref{prop:approx_comp}
ensures there exists $C_1\in (0,\infty)$ such that
for all $h\in (0,\ol h)$ and $y\in I^\circ$ it holds that 
\begin{equation}
a_h(y)\le \sqrt{(\gamma+1)k_1(1+k_2(2|y|+ C_1)^2)}\sqrt{h}.
\end{equation}
This
proves~\eqref{eq:lin_growth_sf}.
\end{proof}

\section{Condition~(C) implies finite moments}\label{sec:fin_mom}
We first establish moment bounds and a kind of temporal regularity for the general diffusion under Condition~(C).

\begin{theo}\label{prop:moment_bound_y}
Assume that Condition~(C) is satisfied and 
let $p\in [1,\infty)$ and $T\in (0,\infty)$.
Then there exists $ C\in (0,\infty)$ such that
for all $y\in I^\circ$ and $s<t$ in $[0,T]$ it holds
\begin{equation}\label{eq:time_reg_mom}
\left\|\sup_{u\in [s,t]}|Y_u-Y_s|\right\|_{L^p(P_y)}
\le C (1+k_2|y|) \sqrt{t-s}.
\end{equation}
In particular, for all $p\in[1,\infty)$ and $T\in(0,\infty)$
there exists $D\in(0,\infty)$ such that for all $y\in I^\circ$ we have
\begin{equation}\label{eq:24022019a1}
\left\|\sup_{u\in [0,T]}|Y_u-y|\right\|_{L^p(P_y)}\le D(1+k_2|y|).
\end{equation}
\end{theo}

\begin{proof}
Throughout the proof we assume without loss of generality that $p\in [2,\infty)$
(the statement for $p\in[1,2)$ follows via monotonicity of $L^p$ norms). 
We fix $y\in I^\circ$ and denote by $C_i$, $i\in \N$,
constants that do not depend on $y$.
We use the fact that $Y$ can be written as a time-changed Brownian motion. To be more precise, by extending the probability space (if necessary),
we can find a Brownian motion $W$
starting in $y$ under $P_y$ such that
\begin{equation}\label{eq:25022019a1}
Y_t=W_{\gamma(t)},\quad\text{for all }t\in[0,\infty),
\end{equation}
with the random time change $\gamma\colon[0,\infty)\to[0,\infty)$
being the right inverse of $A$, i.e.,
\begin{equation}
\gamma(t)=\inf\{s\in [0,\infty)\colon A(s)>t\},\quad t \in  [0,\infty),
\end{equation}
where the increasing process $A\colon[0,\infty)\to[0,\infty]$
is given by the formula
\begin{equation}
A(t) = \frac12 \int_{I} L^x_t \, m(dx),\quad t\in [0,\infty),
\end{equation}
and $(L^x_t)$, $t \in [0,\infty)$, $x \in \IR$, denotes the the local time of~$W$
(see Theorem~V.47.1 in \cite{RogersWilliams}
and notice that our speed measure is twice as large
as the speed measure in \cite{RogersWilliams}).
Furthermore, $A$ is strictly increasing and continuous
on $[0,H_{l,r}(W)]\cap[0,\infty)$
and infinite on $(H_{l,r}(W),\infty)$;
$\gamma$ is continuous on $[0,\infty)$,
strictly increasing on $[0,H_{l,r}(Y)]$
(notice that $H_{l,r}(Y)=A(H_{l,r}(W))$)
and, on the event $\{H_{l,r}(Y)<\infty\}$,
$\gamma$ is constant on $[H_{l,r}(Y),\infty)$
(recall that accessible boundaries are assumed to be absorbing).
It follows from Condition~(C) that 
for all $s<t$ in $[0,H_{l,r}(W)]$ it holds
\begin{equation}\label{eq:lb_A2}
A(t) - A(s)  = \frac12 \int_{I} (L^x_t - L^x_s) m(dx)
\ge  \int_{\IR} 1_{I}(x)\frac{(L^x_t - L^x_s)}{k_1(1+k_2x^2)}\, dx =  \int_s^t \frac{1}{k_1(1+k_2W_u^2)}\,du,
\end{equation}
where the last equality follows from the occupation time formula.
Let $f \in C^1([0,\infty),[0,\infty))$ be the function given by $f(t)=\int_0^t k_1(1+k_2W^2_{\gamma(s)})ds$, $t\in [0,\infty)$.
Then it follows with the change of variables in Stieltjes integrals
for all $t\in [0,H_{l,r}(W)]$ that
$$
\int_0^{A(t)} k_1(1+k_2W^2_{\gamma(s)})\,ds=f(A(t))=\int_0^t f'(A(s))\,dA(s)=\int_0^t k_1(1+k_2W_s^2)\,dA(s).
$$
This and \eqref{eq:lb_A2} imply for all $s<t$ in $[0,H_{l,r}(Y)]$ that
$$
\gamma(t)-\gamma(s)
\le\int_{\gamma(s)}^{\gamma(t)} k_1(1+k_2W_u^2)\,dA(u)
=\int_s^{t} k_1(1+k_2W^2_{\gamma(u)})\,du
=\int_s^{t} k_1(1+k_2Y_u^2)\,du
$$
(notice that $A(\gamma(t))=t$ because $t\le H_{l,r}(Y)$).
As $\gamma$ is constant on $[H_{l,r}(Y),\infty)$, the latter estimate
\begin{equation}\label{eq:ub_gamma2}
\gamma(t)-\gamma(s)
\le\int_s^{t} k_1(1+k_2Y_u^2)\,du
\end{equation}
holds, in fact, for all $s<t$ in $[0,\infty)$.

Next, let $\widehat \tau$ be an $(\cF_t)_{t\in [0,\infty)}$-stopping time. Then it holds 
for all $t\in [0,\infty)$ that
\begin{equation}
E_y[\sup_{s\in [0,t]}|Y^{\widehat \tau}_s-y|^p]
= 
E_y[\sup_{s\in [0,t\wedge \widehat \tau]}|Y_s-y|^p]
=
E_y[\sup_{s\in [0,\gamma(t\wedge \widehat \tau)]}|W_s-y|^p].
\end{equation}
This and the Burkholder-Davis-Gundy inequality show that there exists $C_1\in (0,\infty)$ (independent of $\widehat \tau$ and $y$) such that for all $t\in [0,\infty)$ it holds
\begin{equation}
E_y[\sup_{s\in [0,t]}|Y^{\widehat \tau}_s-y|^p]= E_y[\sup_{s\in [0,\gamma(t\wedge \widehat \tau)]}|W_s-y|^p]
\le 
C_1 E_y[\gamma(t\wedge \widehat \tau)^{\frac{p}{2}}].
\end{equation}
Then it follows with \eqref{eq:ub_gamma2} for all $t\in [0,\infty)$ that
\begin{equation}
\begin{split}
E_y[\sup_{s\in [0,t]}|Y^{\widehat \tau}_s-y|^p]
&\le 
C_1E_y\left [\left(  \int_0^{t\wedge \widehat \tau} k_1(1+k_2Y_s^2)ds \right)^{\frac{p}{2}}\right]\\
&\le 
C_1 k_1^{\frac{p}{2}} E_y\left [\left(t+2k_2y^2 t+  2k_2\int_0^{t}(Y_s^{\widehat \tau}-y)^2ds \right)^{\frac{p}{2}}\right]\\
&\le 
2^{\frac{p}{2}-1}C_1k_1^{\frac{p}{2}}\left((t+2k_2y^2t)^{\frac{p}{2}}+(2k_2)^{\frac{p}{2}} E_y\left [
\left(\int_0^{t}(Y_s^{\widehat \tau}-y)^2ds \right)^{\frac{p}{2}}\right]\right).
\end{split}
\end{equation}
Combining this with Jensen's inequality proves for all $t\in [0,T]$ that
\begin{equation}\label{eq:ub_mom_Y_aux}
\begin{split}
E_y[\sup_{u\in [0,t]}|Y^{\widehat \tau}_u-y|^p]
&\le 
2^{\frac{p}{2}-1}C_1k_1^{\frac{p}{2}}\left((t+2k_2y^2 t)^{\frac{p}{2}}+(2k_2)^{\frac{p}{2}}t^{\frac{p}{2}-1}\int_0^t E_y\left [
|Y_s^{\widehat \tau}-y|^p \right] ds\right)\\
&\le 
2^{\frac{p}{2}-1}C_1k_1^{\frac{p}{2}}\left((T+2k_2y^2 T)^{\frac{p}{2}}+(2k_2)^{\frac{p}{2}}T^{\frac{p}{2}-1}\int_0^t E_y\left [
\sup_{u\in[0,s]}
|Y_u^{\widehat \tau}-y|^p ds\right]\right).
\end{split}
\end{equation}
We choose $\widehat \tau = \widehat \tau_n=\inf\{s\in [0,\infty): |Y_s-y|\ge n\}$, $n\in \N$. Then it holds for all $n\in \N$, $t\in [0,T]$ that
$\int_0^{t}E_y [\sup_{u\in [0,s]}|Y^{\widehat \tau_n}_u-y|^p]ds\le n^p \, T<\infty$. Applying Gronwall's inequality to~\eqref{eq:ub_mom_Y_aux} shows for all $n\in \N$ and $t\in [0,T]$ that
\begin{equation}
E_y[\sup_{s\in [0,t]}|Y^{\widehat \tau_n}_s-y|^p]
\le C_1k_1^{\frac{p}{2}}2^{\frac{p}{2}-1}(T+2k_2y^2T)^{\frac{p}{2}}
e^{C_1k_1^{\frac{p}{2}}2^{\frac{p}{2}-1}(2k_2)^{\frac{p}{2}}T^{\frac{p}{2}-1}t}.
\end{equation}
Since $\widehat \tau_n\nearrow \infty$ as $n\to \infty$, it follows with monotone convergence that
\begin{equation}
E_y[\sup_{s\in [0,T]}|Y_s-y|^p]
=\lim_{n\to \infty}
E_y[\sup_{s\in [0,T]}|Y^{\widehat \tau_n}_s-y|^p]
\le C_1k_1^{\frac{p}{2}}2^{\frac{p}{2}-1}(T+2k_2y^2T)^{\frac{p}{2}}
e^{C_1k_1^{\frac{p}{2}}2^{\frac{p}{2}-1}(2k_2)^{\frac{p}{2}}T^{\frac{p}{2}}}.
\end{equation}
To summarize, this proves that there exists $C_2\in (0,\infty)$ (independent of $y\in I^\circ$)
such that it holds
\begin{equation}\label{eq:moment_bound_y}
\begin{split}
\left\|\sup_{s\in [0,T]}|Y_s-y|\right\|_{L^p(P_y)}\le C_2(1+k_2|y|).
\end{split}
\end{equation}

Now note that for all $s<t$ in $[0,\infty)$ it holds that
\begin{equation}
E_y\left[\sup_{u\in [s,t]}|Y_u-Y_s|^p\right]
= 
E_y\left[\sup_{u\in [\gamma(s),\gamma(t)]}|W_u-W_{\gamma(s)}|^p\right].
\end{equation}
This and the Burkholder-Davis-Gundy inequality show that there exists $C_3\in (0,\infty)$ such that for all $s<t$ in $[0,\infty)$ it holds
\begin{equation}
E_y\left[\sup_{u\in [s,t]}|Y_u-Y_s|^p\right]
\le 
C_3 E_y\left[|\gamma(t)-\gamma(s)|^{\frac{p}{2}}\right].
\end{equation}
Then it follows with~\eqref{eq:ub_gamma2}
for all $s<t$ in $[0,\infty)$ that
\begin{equation}
\begin{split}
E_y\left[\sup_{u\in [s,t]}|Y_u-Y_s|^p\right]
&\le 
C_3E_y\left [\left( \int_s^{t} k_1(1+k_2Y_u^2)\,du \right)^{\frac{p}{2}}\right]\\
&\le 
C_3k_1^{\frac{p}{2}}E_y\left [\left(1+2k_2y^2+2k_2\sup_{u\in [s,t]}|Y_u-y|^2\right)^{\frac{p}{2}}\right](t-s)^{\frac{p}{2}}.
\end{split}
\end{equation}
Together with~\eqref{eq:moment_bound_y}, this
proves~\eqref{eq:time_reg_mom}.
Finally, \eqref{eq:24022019a1} immediately follows
from~\eqref{eq:time_reg_mom}.
\end{proof}

The next result establishes uniform in $h$ moment bounds for the approximations.

\begin{theo}\label{prop:moment_bound_approx}
Suppose that Condition~(C) is satisfied
and that \eqref{eq:gamma_cond} holds true.
Let $p\in [1,\infty)$ and $T\in (0,\infty)$.
Then there exists $C\in (0,\infty)$ such that 
for all $y\in I^\circ$ it holds that
\begin{equation}\label{eq:moment_bound_approx}
\sup_{h\in (0,\ol h)}\left\|\sup_{s\in [0,T]}| X^{h,y}_{s}-y| \right\|_{L^p(P)}
\le\sup_{h\in (0,\ol h)}\left \|\max_{k\in \{0,\ldots, \lceil T/h \rceil\}}|X^{h,y}_{kh}-y| \right\|_{L^p(P)}\le C(1+k_2 |y|).
\end{equation}
In particular, \eqref{eq:moment_bound_approx}
holds under Condition~(C)
and Condition~(A$\lambda$)
with an arbitrary $\lambda\in(0,\infty)$.
\end{theo}

\begin{proof}
Throughout the proof we assume without loss of generality that $p\in [2,\infty)$.  We fix $y\in I^\circ$ and denote by $C_i$, $i\in \N$,
constants that do not depend on $y$. For all
$h\in (0,\ol h)$ let $M^{h,y}_n=\max_{k\in \{0,\ldots,n\}}|X^{h,y}_{kh}-y|$,  $n\in \N_0$. 
The first  inequality in~\eqref{eq:moment_bound_approx} holds
by the construction of the continuous-time process $X^{h,y}$.
The fact that for every $h\in (0,\ol h)$ the process $(X^{h,y}_{kh})_{k\in\bbN_0}$ is a martingale, the Burkholder-Davis-Gundy inequality and the definition of
$(X^{h,y}_{kh})_{k\in\bbN_0}$
ensure that there exists $C_1\in (0,\infty)$ (independent of $y$) such that for all 
$h\in (0,\ol h)$ and $n\in \N$ it holds that
\begin{equation}
E[|M^{h,y}_n|^p]\le C_1 E\left[\left(\sum_{k=0}^{n-1}(X^{h,y}_{(k+1)h}-X^{h,y}_{kh})^2\right)^{\frac{p}{2}} \right]
=
C_1 E\left[\left(\sum_{k=0}^{n-1}(a_h(X^{h,y}_{kh}))^2\right)^{\frac{p}{2}} \right].
\end{equation}
This together with Proposition~\ref{prop:cond_d_lin_growth} shows that
 there exists a constant $C_2\in (0,\infty)$ such that
 for all 
 $h\in (0,\ol h)$ and $n\in \N$ it holds
 \begin{equation}
 \begin{split}
 E[|M^{h,y}_n|^p]&\le
 C_2 E\left[\left(\sum_{k=0}^{n-1}h(1+k_2(X^{h,y}_{kh})^2)\right)^{\frac{p}{2}} \right]\\
 &=
 C_2 E\left[\left(hn+k_2h\sum_{k=0}^{n-1}(X^{h,y}_{kh})^2\right)^{\frac{p}{2}} \right]\\
 &\le 
 C_2 E\left[\left((1+2k_2y^2)hn+2k_2h\sum_{k=0}^{n-1}(X^{h,y}_{kh}-y)^2\right)^{\frac{p}{2}} \right]\\
 &\le 
 C_22^{\frac{p}{2}-1}\left[ ((1+2k_2y^2)hn)^{\frac{p}{2}}+(2k_2h)^{\frac{p}{2}}E\left[\left(\sum_{k=0}^{n-1}(X^{h,y}_{kh}-y)^2\right)^{\frac{p}{2}} \right]\right].
 \end{split}
 \end{equation}
 Combining this with Jensen's inequality proves for all 
 $h\in (0,\ol h)$, $n\in \{1,\ldots, \lceil T/h \rceil\}$  that
 \begin{equation}
 \begin{split}
 E[|M^{h,y}_n|^p]
 &\le 
 C_2 2^{\frac{p}{2}-1}\left[ ((1+2k_2y^2)(T+1))^{\frac{p}{2}}+(2k_2h)^{\frac{p}{2}}n^{\frac{p}{2}-1}E\left[\sum_{k=0}^{n-1}|X^{h,y}_{kh}-y|^p \right]\right]\\
 &\le 
 C_22^{\frac{p}{2}-1}\left[ ((1+2k_2y^2)(T+1))^{\frac{p}{2}}+(2k_2)^{\frac{p}{2}}(T+1)^{\frac{p}{2}-1}h\sum_{k=0}^{n-1}E\left[|M^{h,y}_{k}|^p \right]\right].
 \end{split}
  \end{equation}
  It follows from Gronwall's inequality that for all 
 $h\in (0,\ol h)$, $n\in \{1,\ldots, \lceil T/h \rceil\}$ it holds
  \begin{equation}
  E[|M^{h,y}_n|^p]\le C_22^{\frac{p}{2}-1}((1+2k_2y^2)(T+1))^{\frac{p}{2}}
  \left(1+ C_22^{\frac{p}{2}-1}(2k_2)^{\frac{p}{2}}(T+1)^{\frac{p}{2}-1}h \right)^{n-1}.
  \end{equation}
  This implies that
  \begin{equation}
  \begin{split}
& \sup_{h\in (0,\ol h)}E\left [\max_{k\in \{0,\ldots, \lceil T/h \rceil\}}|X^{h,y}_{kh}-y|^p \right]\\
&\le 
 C_22^{\frac{p}{2}-1}((1+2k_2y^2)(T+1))^{\frac{p}{2}}
  \left[ \sup_{h\in (0,\ol h)}
  \left(1+ C_22^{\frac{p}{2}-1}(2k_2)^{\frac{p}{2}}(T+1)^{\frac{p}{2}-1}h \right)^{\frac{T}{h}}\right]
  <\infty,
  \end{split}
  \end{equation}
  which completes the proof.
\end{proof}

\section{Embedding the approximations into the general diffusion}\label{sec:embedding}
The proofs of
Theorems \ref{thm:wasserstein_term} and~\ref{thm:wasserstein_path}
rely on embeddings of the Markov chains $(X^{h,y}_{kh})_{k\in \N_0}$ into the diffusion $Y$ with sequences of stopping times. 
In Section~3 in \cite{aku2018cointossing} (see, in particular, Proposition~3.1 therein) we construct for all $h\in (0,\ol h)$ a sequence of stopping times $(\tau^h_k)_{k\in \N_0}$ such that
for all
$h\in (0,\ol h)$ and $y\in I^\circ$ we have 
\begin{equation}\label{eq:same_dist}
\Law_{P_y}
\left(Y_{\tau^h_k}; k\in \Znn \right)
=\Law_P
\left(X^{h,y}_{kh}; k\in \Znn\right).
\end{equation}
In the case, where both boundary points are inaccessible, the construction of the stopping times is straightforward: for all $h\in (0,\ol h)$ set $\tau^h_0=0$ and then
recursively define $\tau^h_{k+1}$ as the first time $t\ge \tau^h_{k}$ such 
that $|Y_t- Y_{\tau^h_k}|=a_h(Y_{\tau^h_k})$. In the case of accessible
(then absorbing) boundary points 
the stopping times are deterministically extended
on reaching an absorbing boundary
whenever $Y_{\tau^h_k}\notin I_h$.
For more details, see Section~3 in \cite{aku2018cointossing}.

Remark~1.2 in \cite{aku2018cointossing} shows that for every
$y\in I^\circ$ and $a\in[0,\infty)$ such that $y\pm a\in I$
the expected time it takes $Y$ started in $y$ to leave the interval 
$(y-a,y+a)$ is equal to $\frac{1}{2}\int_{(y-a,y+a)}(a-|u-y|)\, m(du)$.
Therefore, Condition~(A$\lambda$) ensures, roughly speaking, for all $k\in \N_0$ that the expected time lag $\tau^h_{k+1}-\tau^h_k$ deviates from $h$ by an order of $h^{1+\lambda}$. 
The next result concludes that the global temporal error $|\tau^h_k-kh|$ is of order $h^{\min\{\frac{1}{2},\lambda\}}$ under Condition~(A$\lambda$).
It is worth noting that, in contrast to many preceding results,
we do not need Condition~(C) here.

\begin{propo}\label{cor:26112018a1}
Let $p \in [1,\infty)$, $T\in (0,\infty)$, $\lambda\in (0,\infty)$ and assume that Condition~(A$\lambda$) is satisfied.
Then there exists $C\in (0,\infty)$ such that
for all $y \in I^\circ$ and $h\in (0,\ol h)$ it holds that
\begin{equation}\label{eq:glob_temp_error}
\left\|\sup_{k\in \{1,\ldots, \lfloor T/h \rfloor \}}\left|\tau^h_k-kh\right|\right\|_{L^p(P_y)}
\le C
 h^{\min\{\frac{1}{2},\lambda\}}.
\end{equation}
\end{propo}

\begin{remark}
Note that the order of convergence of the right-hand side of \eqref{eq:glob_temp_error} does not become better than $\sqrt{h}$ for $\lambda\in [\frac{1}{2},\infty)$. We refer the reader to Equation~\eqref{eq:lb_temp_error} in the proof of Proposition~\ref{prop:lb_l4} below to see that $\sqrt{h}$ is the optimal order that, in general, cannot be improved even if 
$
\frac{1}{2}\int_{(y-a_h(y),y+a_h(y))} (a_h(y)-|u-y|)\,m(du)
=h
$ for all $y\in I^\circ$, $h\in (0,\ol h)$,
i.e., Condition~(A$\lambda$) is satisfied for all $\lambda \in (0,\infty)$.
\end{remark}

\begin{proof}
We introduce the time lags between two consecutive stopping times $\rho^h_k=\tau^h_{k}-\tau^h_{k-1}$, $h\in (0,\ol h)$, $k\in \N$.
It follows from Proposition~5.2 and Remark~1.2 in \cite{aku2018cointossing}
that there exists a constant $C_1\in (0,\infty)$
such that for all $y \in I^\circ$ and $h\in (0,\ol h)$
it holds that
\begin{equation}\label{eq:lalphamart}
\begin{split}
&\left\|\sup_{k\in \{1,\ldots, \lfloor T/h\rfloor \}}\left|\tau^h_k-\sum_{n=1}^k E[\rho^h_n|\mathcal F_{\tau^h_{n-1}}]\right|\right\|_{L^p(P_y)}\\
&\le C_1\sqrt{T}
\left(\sqrt{h}+
\frac{\sup_{z\in I}\left( E_z[H_{z-a_h(z), z + a_h(z)}(Y)]\right)}{\sqrt{h}}
\right)\\
&= C_1\sqrt{T}
\left(\sqrt{h}+
\frac{\frac12
\sup_{z\in I}\left( \int_{(z-a_h(z),z+a_h(z)}(a_h(z)-|u-z|)\,m(du)\right)
}{\sqrt{h}}
\right).
\end{split}
\end{equation}
It follows from the definition of the sets $I_h$, $h\in (0,\ol h)$, that for all $h\in (0,\ol h)$ it holds
$$
\frac12\sup_{z\in I\setminus I_h}
 \int_{(z-a_h(z),z+a_h(z))}(a_h(z)-|u-z|)\,m(du)\le h. 
$$
Condition~(A$\lambda$) implies for all $h\in (0,\ol h)$
$$
\frac12 \sup_{z\in I_h}\int_{(z-a_h(z),z+a_h(z))}(a_h(z)-|u-z|)\,m(du)\le Kh^{1+\lambda}+h\le (K+1)h.
$$
Hence, we obtain that for all $y \in I^\circ$ and $h\in (0,\ol h)$
\begin{equation}\label{eq:lalphamart2}
\begin{split}
&\left\|\sup_{k\in \{1,\ldots, \lfloor T/h\rfloor \}}\left|\tau^h_k-\sum_{n=1}^k E[\rho^h_n|\mathcal F_{\tau^h_{n-1}}]\right|\right\|_{L^p(P_y)}
\le C_1\sqrt{T}(K+2) \sqrt{h}.
\end{split}
\end{equation}
Moreover, Proposition~5.3,  Remark~1.2 in \cite{aku2018cointossing} and Condition~(A$\lambda$) imply
that
for all $y\in I^\circ$ and $h\in (0,\ol h)$ it holds that
\begin{equation}\label{eq:lalphapred}
\begin{split}
&\left\|\sup_{k\in \{1,\ldots, \lfloor T/h \rfloor \}}\left|\left(\sum_{n=1}^k E[\rho^h_n|\mathcal F_{\tau^h_{n-1}}]\right)-kh\right|\right\|_{L^p(P_y)}
\le \frac{T}{h} \sup_{z\in I_h}
\left|
E_z[H_{z-a_h(z),z+a_h(z)}(Y)]
-h
\right|\\
&=
 \frac{T}{h} \sup_{z\in I_h}
\left| \frac12
 \int_{(z-a_h(z),z+a_h(z)}(a_h(z)-|u-z|)\,m(du)
-h \right| \le TK h^{\lambda}.
\end{split}
\end{equation}
Combining \eqref{eq:lalphamart2} and \eqref{eq:lalphapred} proves that
for all $y\in I^\circ$ and $h\in (0,\ol h)$ it holds that
\begin{equation*}
\begin{split}
\left\|\sup_{k\in \{1,\ldots, \lfloor T/h \rfloor \}}\left|\tau^h_k-kh\right|\right\|_{L^p(P_y)}
&\le TK h^{\lambda}+ C_1\sqrt{T}(K+2) \sqrt{h}\\
 &\le  \left( TK+C_1\sqrt{T}(K+2) \right) h^{\min\{\frac{1}{2},\lambda\}}.
\end{split}
\end{equation*}
This concludes the proof.
\end{proof}

In addition to the moment estimates for general diffusions
presented in Section~\ref{sec:fin_mom}
we need a moment estimate up to the stopping times
$\tau^h_{\lfloor T/h\rfloor}$ for $T\in(0,\infty)$.

\begin{propo}\label{prop:moment_bound_til_tau}
Assume that Condition~(C) holds and that \eqref{eq:gamma_cond} is satisfied.
Let $p\in [1,\infty)$ and $T\in (0,\infty)$.
Then there exists $C\in (0,\infty)$ such that for all $y\in I^\circ$ it holds
\begin{equation}\label{eq:24022019a2}
\sup_{h\in (0,\ol h)}
\left\|\sup_{s\in [0,\tau^h_{\lfloor T/h \rfloor}] }|Y_{s}| \right\|_{L^p(P_y)}\le 
C(1+|y|).
\end{equation}
In particular, \eqref{eq:24022019a2}
holds under Condition~(C)
and Condition~(A$\lambda$)
with an arbitrary $\lambda\in(0,\infty)$.
\end{propo}

\begin{proof}
First note that by construction of the stopping times 
$\tau^h_k$, $h\in (0,\ol h)$, $k\in\N_0$,
it holds $P_y$-almost surely for all $y\in I^\circ$
that for all $h\in (0,\ol h)$ and $k\in \N_0$ 
\begin{equation}
\sup_{s\in [\tau^h_k,\tau^h_{k+1}]}|Y_s|\le \max\{ |Y_{\tau^h_k}+a_h(Y_{\tau^h_k})|,|Y_{\tau^h_k}-a_h(Y_{\tau^h_k})|\}
= |Y_{\tau^h_k}|+a_h(Y_{\tau^h_k}).
\end{equation}
This together with 
Proposition~\ref{prop:approx_comp} 
ensures that 
there exists $C_1\in (0,\infty)$ (independent of $y\in I^\circ$) such that
for all
$h\in (0,\ol h)$ and $k\in \N_0$ it holds 
\begin{equation}
\sup_{s\in [\tau^h_k,\tau^h_{k+1}]}|Y_s|\le
C_1+2|Y_{\tau^h_k}|.
\end{equation}
Combining this with Theorem~\ref{prop:moment_bound_approx}
and~\eqref{eq:same_dist}
shows that there exists $C_2\in (0,\infty)$ such that for all $y\in I^\circ$
it holds 
\begin{equation}
\begin{split}
\sup_{h\in (0,\ol h)}
\left\|\sup_{s\in [0,\tau^h_{\lfloor T/h \rfloor}] }|Y_{s}| \right\|_{L^p(P_y)}
&=
\sup_{h\in (0,\ol h)}
\left\|\max_{k\in\{0,\ldots,\lfloor T/h \rfloor-1\}}\sup_{s\in [\tau^h_k,\tau^h_{k+1}] }|Y_{s}| \right\|_{L^p(P_y)}\\
&\le 
C_1+
2\sup_{h\in (0,\ol h)}
\left\|\max_{k\in\{0,\ldots,\lfloor T/h \rfloor-1\}}|Y_{\tau^h_k}| \right\|_{L^p(P_y)}
\\
&\le
C_1+
2\left(
C_2(1+k_2|y|)
+|y|
\right),
\end{split}
\end{equation}
which completes the proof.
\end{proof}

\section{Rate of convergence for time marginals}
\label{sec:rate}
In this section, we prove Theorem~\ref{thm:wasserstein_term}. 
Throughout the section let $(\tau^h_k)_{k\in \N_0}$, $h\in (0,\ol h)$, be the embedding stopping times introduced
in Section~\ref{sec:embedding}. A central step in our approach is to establish $L^p$ bounds between the random variables $Y_T$ and $Y_{\tau^h_{\lfloor T/h \rfloor}}$.
These are provided in the next result.

\begin{theo}\label{thm:rate_term}
Suppose that Condition~(C) is satisfied
and that there exists $\lambda \in (0,\infty)$
such that Condition~(A$\lambda$) holds.
Let $T\in (0,\infty)$ and $p\in [1,\infty)$.
Then there exists a constant $C(p,T) \in (0,\infty)$ such that for all $y\in I^\circ$ and $h\in (0,\ol h)$ we have 
\begin{equation}\label{eq:rate_term}
\left\|Y_{\tau^h_{\lfloor T/h \rfloor}}-Y_T\right\|_{L^p(P_y)}
\le C(p,T)(1+k_2|y|) h^{\min\{\frac14,\frac{\lambda}{2}\}}.
\end{equation}
\end{theo}


\begin{proof}
Throughout the proof we use the notation
from the proof of Theorem~\ref{prop:moment_bound_y},
in particular, the fact that $Y$ can be written
as a time-changed Brownian motion (recall~\eqref{eq:25022019a1}).
We fix $y\in I^\circ$ and denote by $C_i\in (0,\infty)$, $i\in \N$, constants that are independent of $y$.
First note that for all $h\in (0,\ol h)$ it holds that
\begin{equation}
E_y[|Y_{\tau^h_{\lfloor T/h \rfloor}}-Y_T|^p]
= 
E_y[|W_{\gamma(\tau^h_{\lfloor T/h \rfloor})}-W_{\gamma(T)}|^p].
\end{equation}
This and the Burkholder-Davis-Gundy inequality show that there exists $C_1\in (0,\infty)$ such that for all $h\in (0,\ol h)$ it holds
\begin{equation}
E_y[|Y_{\tau^h_{\lfloor T/h \rfloor}}-Y_T|^p]
\le 
C_1 E_y[|\gamma(\tau^h_{\lfloor T/h \rfloor})-\gamma(T)|^{\frac{p}{2}}].
\end{equation}
Then it follows with \eqref{eq:ub_gamma2} and the Cauchy-Schwarz inequality
that there exists $C_2\in(0,\infty)$ such that
for all $h\in (0,\ol h)$ we have
\begin{equation}
\begin{split}
E_y[|Y_{\tau^h_{\lfloor T/h \rfloor}}-Y_T|^p]
&\le 
C_2E_y\left [\left( \int_{\tau^h_{\lfloor T/h \rfloor}\wedge T}^{\tau^h_{\lfloor T/h \rfloor}\vee T} (1+k_2Y_u^2)du \right)^{\frac{p}{2}}\right]\\
&\le 
C_2E_y\left [(1+k_2\sup_{u\in [0,\tau^h_{\lfloor T/h \rfloor}\vee T]}Y_u^2)^{\frac{p}{2}}|\tau^h_{\lfloor T/h \rfloor}-T|^{\frac{p}{2}}\right]\\
&\le 
C_2\
\sqrt{E_y\left [\left(1+k_2\sup_{u\in [0,\tau^h_{\lfloor T/h \rfloor}\vee T]}Y_u^2\right)^{p}\right]E_y\left [|\tau^h_{\lfloor T/h \rfloor}-T|^{p}\right]}.
\end{split}
\end{equation}
The triangle inequality ensures for all $h\in (0,\ol h)$ that
\begin{equation}\label{eq:sing_point_aux1}
\begin{split}
&\left\|Y_{\tau^h_{\lfloor T/h \rfloor}}-Y_T\right\|_{L^p(P_y)}\\
&\le C_2^{\frac{1}{p}} \sqrt{\left(1+k_2\left\|\sup_{u\in [0,\tau^h_{\lfloor T/h \rfloor}\vee T]}Y_u^2 \right\|_{L^p(P_y)} \right) 
\left(
\left\|  \tau^h_{\lfloor T/h \rfloor}-h \left\lfloor \frac{T}{h} \right \rfloor   \right\|_{L^p(P_y)}
+  T-  h\left\lfloor \frac{T}{h}  \right \rfloor 
\right)
 }\\
 &\le C_2^{\frac{1}{p}} \sqrt{ \left(1+k_2\left\|\sup_{u\in [0,\tau^h_{\lfloor T/h \rfloor}\vee T]}|Y_u| \right\|^2_{L^{2p}(P_y)} \right) 
\left(
\left\|  \tau^h_{\lfloor T/h \rfloor}-h \left\lfloor \frac{T}{h} \right \rfloor   \right\|_{L^p(P_y)}
+ h
\right)
 }.
 \end{split}
\end{equation}
Theorems \ref{prop:moment_bound_y}
and~\ref{prop:moment_bound_til_tau} ensure that
there exists $C_3 \in (0,\infty)$ such that
\begin{equation}\label{eq:sing_point_aux2}
\sup_{h\in (0,\ol h)}
\left\|\sup_{u\in [0,\tau^h_{\lfloor T/h \rfloor}\vee T]}|Y_u| \right\|_{L^{2p}(P_y)}
\le 
C_3(1+|y|).
\end{equation}
Moreover, Proposition~\ref{cor:26112018a1} implies that there exists $C_4 \in (0,\infty)$ such that for all $h\in (0,\ol h)$ it holds that
\begin{equation}\label{eq:sing_point_aux3}
\left\|  \tau^h_{\lfloor T/h \rfloor}-h \left\lfloor \frac{T}{h} \right \rfloor   \right\|_{L^p(P_y)}
\le C_4 h^{\min\{\frac12,\lambda\}}.
\end{equation}
Combining \eqref{eq:sing_point_aux1}, \eqref{eq:sing_point_aux2} and 
\eqref{eq:sing_point_aux3}
entails that there exists $C_5\in (0,\infty)$ such that
\begin{equation}
\left\|Y_{\tau^h_{\lfloor T/h \rfloor}}-Y_T\right\|_{L^p(P_y)}
\le C_5(1+k_2|y|) h^{\min\{\frac14,\frac{\lambda}{2}\}}.
\end{equation}
This completes the proof.
\end{proof}

\begin{proof}[Proof of Theorem~\ref{thm:wasserstein_term}]
Throughout the proof let $p\in [1,\infty)$ and $T\in (0,\infty)$.
By Theorems \ref{prop:moment_bound_y}
and~\ref{prop:moment_bound_approx},
the random variables $Y_T$,
$X^{h,y}_{h\lfloor T/h \rfloor}$
and $X^{h,y}_T$,
$h\in (0,\ol h)$, admit finite $p$-th moments.
For every $h\in(0,\ol h)$ we define the random variable
\begin{equation}\label{eq:24022019b1}
Y^h_T=Y_{\tau^h_{\lfloor T/h\rfloor}}
+(T/h-\lfloor T/h\rfloor)
(Y_{\tau^h_{\lfloor T/h\rfloor+1}}-Y_{\tau^h_{\lfloor T/h\rfloor}})
\end{equation}
and notice that,
by \eqref{eq:same_dist} and~\eqref{eq:13112017a1},
for every $h\in(0,\ol h)$ and $y\in I^\circ$
it holds
$$
\Law_{P_y}
(Y_{\tau^h_{\lfloor T/h \rfloor}})
=\Law_P
(X^{h,y}_{h\lfloor T/h \rfloor})
\quad\text{and}\quad
\Law_{P_y}
(Y^h_T)
=\Law_P
(X^{h,y}_T).
$$
In particular,
the random vector $(Y_{\tau^h_{\lfloor T/h \rfloor}},Y_T)$
(resp.\ $(Y^h_T,Y_T)$)
is a coupling between the laws of $X^{h,y}_{h\lfloor T/h \rfloor}$ and $Y_T$
(resp.\ $X^{h,y}_T$ and $Y_T$).
Therefore, Statement~\eqref{eq:24022019c1}
of Theorem~\ref{thm:wasserstein_term}
follows directly from Theorem~\ref{thm:rate_term}.
Next, applying~\eqref{eq:24022019b1},
Proposition~\ref{prop:cond_d_lin_growth},
Theorem~\ref{prop:moment_bound_approx}
and~\eqref{eq:same_dist},
we obtain that
there exist constants $C_1\in(0,\infty)$ and $C_2\in(0,\infty)$
such that for all $y\in I^\circ$ and $h\in(0,\ol h)$ it holds
\begin{equation*}
\begin{split}
\left\|Y^h_T-Y_{\tau^h_{\lfloor T/h \rfloor}}\right\|_{L^p(P_y)}
&\le
\left\|Y_{\tau^h_{\lfloor T/h \rfloor+1}}-Y_{\tau^h_{\lfloor T/h \rfloor}}\right\|_{L^p(P_y)}
=
\left\|a_h(Y_{\tau^h_{\lfloor T/h \rfloor}})\right\|_{L^p(P_y)}\\
&\le
C_1\sqrt{h}
\left\|1+k_2|Y_{\tau^h_{\lfloor T/h \rfloor}}|\right\|_{L^p(P_y)}
\le
C_2(1+k_2|y|)\sqrt{h}.
\end{split}
\end{equation*}
Together with Theorem~\ref{thm:rate_term},
this establishes \eqref{eq:24022019c2}
and concludes the proof.
\end{proof}

\subsection*{Optimality of the rate $1/4$}
In this subsection we consider
the EMCEL scheme $a_h=\wh a_h$
(see Example~\ref{ex:25022019a1})
in the situation, where both boundaries are inaccessible
and hence $I_h=I^\circ$ for all $h\in(0,\ol h)$.
Recall that in this case
scale factors $a_h$ are chosen in such a way that
for all 
$y\in I^\circ$, $h\in (0,\ol h)$ it holds
\begin{equation}\label{eq:exact_sf}
\frac{1}{2}\int_{(y-a_h(y),y+a_h(y))} (a_h(y)-|u-y|)\,m(du)=h.
\end{equation}
In particular, Condition (A$\lambda$) is satisfied for all $\lambda \in (0,\infty)$. In this situation Theorem~\ref{thm:rate_term} ensures that 
$Y_{\tau^h_{\lfloor T/h \rfloor}}$ converges to $Y_T$ in $L^p$ at a rate of at least $1/4$. 
We close this section by showing that the convergence rate $1/4$ is the best rate that can be obtained for this $L^p$ convergence.
Indeed, for such a nice Markov process as Brownian motion,
the approximation scheme does not converge faster.
To prove this,
we need the following auxiliary result. 

\begin{lemma}\label{lem:lb_W4tau}
Let $W$ be a Brownian motion starting in $0$,
$\tau$ a square integrable stopping time.
Then it holds that
\begin{equation}\label{eq:lb_W4tau}
E[W^4_\tau]\ge \frac{1}{4} E[\tau^2].
\end{equation}
\end{lemma}

\begin{proof}
First let $\tau$ be bounded. Then it follows with It\^o's formula that
\begin{equation}
E[\tau W_\tau^2]=E\left[\int_0^\tau (2sW_s)dW_s +\int_0^\tau (W_s^2+s)ds \right]
\ge 
E\left[\int_0^\tau s ds \right]=\frac{1}{2}	E[\tau^2].
\end{equation}
This together with Young's inequality implies that
\begin{equation}
\frac{1}{2} E[\tau^2]\le E[\tau W_\tau^2]\le E\left[\frac{1}{4}\tau^2+W_\tau^4\right].
\end{equation}
This proves \eqref{eq:lb_W4tau} for bounded stopping times. Next, let $\tau$ be square integrable. Then it holds for all $n\in \N$ that
\begin{equation}\
E[W^4_{\tau\wedge n}]\ge \frac{1}{4} E[(\tau\wedge n)^2].
\end{equation}
By monotone convergence the right-hand side converges to $\frac{1}{4}E[\tau^2]$ as $n\to \infty$. The Burkholder-Davis-Gundy inequality ensures that $\sup_{s\in [0,\tau]}W^4_{s}$ is integrable and hence dominated convergence ensures that
\begin{equation}
\lim_{n\to \infty}E[W^4_{\tau\wedge n}]=E[W^4_\tau],
\end{equation}
which completes the proof.
\end{proof}

We are now in a position to prove that our scheme applied to $W$ does not converge faster in $L^p$, $p \ge 4$, than at the rate $1/4$.  

\begin{propo}\label{prop:lb_l4}
Assume that $I=\R$ and that the speed measure $m$ satisfies $m(dx)=2\,dx$.
In other words, the general diffusion we are considering
is a Brownian motion $W$ on $\bbR$.
Let scale factors satisfy $a_h(y)=\sqrt{h}$ for all $h\in (0,1)$, $y\in \R$. Then \eqref{eq:exact_sf} is satisfied and for all
$p\in [4,\infty)$, $y\in \R$, $T\in (0,\infty)$ and $h\in (0,1 \wedge T)$ we have
\begin{equation}\label{eq:lb_Lp}
\|W_{\tau^h_{\lfloor T/h \rfloor}}-W_T\|_{L^p(P_y)}
\ge 
\|W_{\tau^h_{\lfloor T/h \rfloor}}-W_T\|_{L^4(P_y)}
\ge 
\left(\frac{(T-h)\Var (H_{-1,1}(W))}{4}\right)^{\frac{1}{4}}h^{\frac{1}{4}}.
\end{equation}
\end{propo}
\begin{proof}
The claim that \eqref{eq:exact_sf} is satisfied follows directly from the assumption that
$a_h(y)=\sqrt{h}$ for all $h\in (0,1)$, $y\in \R$.
To prove \eqref{eq:lb_Lp} we can without loss of generality assume that $y=0$ since $m$ is space invariant. For the remainder of the proof we fix $T\in (0,\infty)$ and $h\in (0,1 \wedge T)$.
By the construction in the beginning of
Section~\ref{sec:embedding}
the stopping times satisfy
for all $k\in \N_0$ that
$\tau^h_0=0$ and $\tau^h_{k+1}=\inf\{t\ge \tau^h_k\colon |W_t-W_{\tau^h_k}|=\sqrt{h}\}$.
It follows, e.g., from Proposition~\ref{cor:26112018a1}, that the stopping time $\tau^h_{\lfloor T/h \rfloor}$ is square integrable.
Then (a conditional version of) Lemma~\ref{lem:lb_W4tau} and the fact that
that $E[\tau^h_{\lfloor T/h \rfloor}]=h\lfloor T/h \rfloor$ show that
\begin{equation}\label{eq:lb_l4_aux1}
\begin{split}
E[|W_{\tau^h_{\lfloor T/h \rfloor}}-W_T|^4]
&=
E[|W_{\tau^h_{\lfloor T/h \rfloor}\vee T}-W_{\tau^h_{\lfloor T/h \rfloor}\wedge T}|^4]
\ge 
\frac{1}{4}E[(\tau^h_{\lfloor T/h \rfloor}\vee T-\tau^h_{\lfloor T/h \rfloor}\wedge T)^2]\\
&=
\frac{1}{4}E[(\tau^h_{\lfloor T/h \rfloor}- T)^2]
\ge \frac14
E\left [\left (\tau^h_{\lfloor T/h \rfloor}- h\left \lfloor \frac{T}{h}\right \rfloor \right )^2 \right].
\end{split}
\end{equation}
Moreover, it holds that $E[\tau^h_{k+1} - \tau^h_k| \cF_{\tau^h_k}] = h$ for all $k\in \N_0$.
In particular, the increments
$\tau^h_{k+1}-\tau^h_k$, $k\in\bbN_0$,
are uncorrelated. Hence we have
\begin{equation}
\begin{split}
E\left [\left (\tau^h_{\lfloor T/h \rfloor}- h\left \lfloor \frac{T}{h}\right \rfloor \right )^2 \right]
&= \Var(\tau^h_{\lfloor T/h \rfloor})
=
\Var\left(\sum_{k=0}^{\lfloor T/h \rfloor-1} (\tau^h_{k+1}-\tau^h_k) \right)\\
&=\sum_{k=0}^{\lfloor T/h \rfloor-1}\Var\left( \tau^h_{k+1}-\tau^h_k \right).
\end{split}
\end{equation}
The assumption that for all $y\in \R$ we have $a_h(y)=\sqrt{h}$ implies that $\tau^h_{k+1}-\tau^h_k  \stackrel{d}{=} H_{-\sqrt{h},\sqrt{h}}(W)$. Therefore we obtain
\begin{equation}
E\left [\left (\tau^h_{\lfloor T/h \rfloor}- h\left \lfloor \frac{T}{h}\right \rfloor \right )^2 \right]= \left \lfloor \frac{T}{h}\right \rfloor
\Var\left( H_{-\sqrt{h},\sqrt{h}}(W) \right).
\end{equation}
The self-similarity of the Brownian motion states that the process $(\sqrt{h}W_{\frac{t}{h}})_{t\in [0,\infty)}$ is again a Brownian motion. This implies that
$H_{-\sqrt{h},\sqrt{h}}(W)\stackrel{d}{=} hH_{-1,1}(W)$. Therefore we obtain that
\begin{equation}\label{eq:lb_temp_error}
E\left [\left (\tau^h_{\lfloor T/h \rfloor}- h\left \lfloor \frac{T}{h}\right \rfloor \right )^2 \right]= h^2\left \lfloor \frac{T}{h}\right \rfloor \Var (H_{-1,1}(W))\ge
h(T-h)\Var (H_{-1,1}(W)).
\end{equation}
Combining this with \eqref{eq:lb_l4_aux1} completes the proof.
\end{proof}

\section{Convergence results on the path space}\label{sec:rate_path}
In this section, we prove Theorems \ref{thm:wasserstein_path},
\ref{cor:26022019a1} and~\ref{th:15042020a1}.
Again let $(\tau^h_k)_{k\in \N_0}$, $h\in (0,\ol h)$, be the embedding stopping times introduced
in Section~\ref{sec:embedding}. 
We define a continuous-time process
$Y^h=(Y^h_t)_{t\in [0,\infty)}$
via linear interpolation
of $(Y_{\tau^h_k})_{k\in\bbN_0}$:
\begin{equation}\label{eq:07082018a2}
 Y^h_t = Y_{\tau^h_{\lfloor t/h \rfloor}} + (t/h - \lfloor t/h \rfloor)
(Y_{\tau^h_{\lfloor t/h \rfloor+1}} - Y_{\tau^h_{\lfloor t/h \rfloor}}),\quad t\in [0,\infty).
\end{equation}
Then it follows from~\eqref{eq:same_dist} that
for all
$h\in (0,\ol h)$ and $y\in I^\circ$ we have
\begin{equation}\label{eq:24022019b2}
\Law_{P_y}
\left(Y^h_t; t\in[0,\infty)\right)
=\Law_P
\left(X^{h,y}_t; t\in[0,\infty)\right)
\end{equation}
(indeed, cf.\ \eqref{eq:07082018a2} with~\eqref{eq:13112017a1}).
We proceed similarly as in Section~\ref{sec:rate} and first establish 
 $L^p$ bounds between the processes $Y$ and~$Y^h$.

\begin{theo}\label{thm:rate_path}
Suppose that Condition~(C) is satisfied and that there exists $\lambda \in (0,\infty)$ such that Condition~(A$\lambda$) holds.
Let $p\in [1,\infty)$, $T\in (0,\infty)$ and
$\varepsilon \in (0,\min\{\frac14,\frac{\lambda}{2}\})$.
Then there exists a constant $C(p,\varepsilon, T) \in (0,\infty)$ such that for all $y\in I^\circ$ and $h\in (0,\ol h)$ we have
\begin{equation}\label{eq:13032019a1}
\left\|\sup_{t\in [0,T]}| Y^h_{t} - Y_{t}|\right\|_{L^p(P_y)}
 \le C(p,\varepsilon,T) (1+k_2|y|) h^{\min\{\frac14,\frac{\lambda}{2}\}-\varepsilon}. 
\end{equation}
\end{theo}

\begin{proof}
To simplify some expressions we treat explicitly only the case
$\lambda\in(0,\frac12]$
(in general, one should replace all instances of $\lambda$
in the proof with $\min\{\frac12,\lambda\}$).
Without loss of generality we assume that
$\varepsilon\in(0,\frac1p\wedge\frac \lambda 2)$.
Throughout the proof we fix $h\in (0,\ol h)$ and $y\in I^\circ$.
We use the notation $C_i(p,\varepsilon,T)\in (0,\infty)$, $i\in \N$,
for constants independent of $h$ and $y$.

First, observe that the Markov inequality and
Proposition~\ref{cor:26112018a1} imply that there exists a constant $C_1(p,T) \in (0,\infty)$ such that
\begin{equation}\label{eq:rate_st_prob}
P_y[\tau^h_{\lfloor T/h \rfloor}\ge T+1]\le P_y\left [\tau^h_{\lfloor T/h \rfloor}-h\left \lfloor \frac{T}{h}\right \rfloor \ge 1\right ]\le E_y\left[\left |\tau^h_{\lfloor T/h \rfloor}-h\left \lfloor \frac{T}{h}\right \rfloor\right |^{p}\right]\le C_1(p,T)h^{\lambda p}.
\end{equation}
This, the Cauchy-Schwarz inequality,
Theorems \ref{prop:moment_bound_y}
and~\ref{prop:moment_bound_approx}
together with~\eqref{eq:24022019b2}
prove that there exists a constant $C_3(p,T)\in (0,\infty)$ such that
\begin{equation}\label{eq:rate_path_aux1}
\begin{split}
&\left\|1_{\{\tau^h_{\lfloor T/h \rfloor}\ge T+1\}}\sup_{t\in [0,T]}| Y^h_{t} - Y_{t}|\right\|_{L^p(P_y)}\le (P_y[\tau^h_{\lfloor T/h \rfloor}\ge T+1])^{\frac{1}{2p}}
\left\|\sup_{t\in [0,T]}| Y^h_{t} - Y_{t}|\right\|_{L^{2p}(P_y)}\\
&\le C_1(p,T)^{\frac{1}{2p}}h^{\frac \lambda 2}\left( \left\|\sup_{t\in [0,T]}| Y^h_{t}-y|\right\|_{L^{2p}(P_y)}
 + \left\|\sup_{t\in [0,T]}|Y_{t}-y|\right\|_{L^{2p}(P_y)}\right)\\
&\le C_3(p,T)(1+k_2|y|)h^{\frac \lambda 2}.
\end{split}
\end{equation}
Next, by the triangle inequality we have that
\begin{equation}\label{eq:rate_triangle}
\sup_{t\in [0,T]}| Y^h_{t} - Y_{t}|\le \sup_{t\in [0,T]}| Y^h_{t} - Y_{\tau^h_{\lfloor t/h \rfloor}}|+\sup_{t\in [0,T]}|Y_{\tau^h_{\lfloor t/h\rfloor}}- Y_{t}|.
\end{equation}
We consider the two terms on the right-hand side of Inequality \eqref{eq:rate_triangle} separately on the 
set $\{\tau^h_{\lfloor T/h \rfloor}< T+1\}$. We start with the latter one. 
By Theorem~\ref{prop:moment_bound_y}, there exists
$C_4(\varepsilon,T)\in (0,\infty)$ such that for all $s,t \in [0,T+1]$ we have
\begin{equation}
E_y[|Y_s-Y_t|^{\frac{1}{\varepsilon}}]\le C_4(\varepsilon,T)(1+k_2|y|)^{\frac1\eps}|s-t|^{\frac{1}{2\varepsilon}}.
\end{equation}
Together with the fact that $0\le \frac{1}{2}-\frac{\varepsilon}{\lambda}<\varepsilon(\frac{1}{2\varepsilon}-1)$,
this allows to apply the Kolmogorov-Chentsov theorem (see, e.g., Theorem~I.2.1 in \cite{RY}) to obtain that 
there exists a constant $C_5(\varepsilon,T)\in (0,\infty)$ such that
\begin{equation}\label{eq:mod_cont}
\left\|\sup_{\substack{s,t\in [0,T+1]\\
s\neq t}}\frac{|Y_s-Y_t|}{|s-t|^{\frac{1}{2}-\frac{\varepsilon}{\lambda} }} \right\|_{L^{\frac{1}{\varepsilon}}(P_y)}
\le C_5(\varepsilon,T)(1+k_2|y|).
\end{equation}
Moreover, it holds that
\begin{equation}
\sup_{t\in [0,T]}|\tau^h_{\lfloor t/h\rfloor}- t|\le 
\sup_{k\in\{0,\ldots,\lfloor T/h \rfloor \}}\sup_{t\in [kh,(k+1)h)}
|\tau^h_{k}- t|
\le 
\sup_{k\in\{0,\ldots,\lfloor T/h \rfloor\}}
\left|\tau^h_{k}- kh\right|+h.
\end{equation}
Since $\frac 12 -\frac{\varepsilon}{\lambda} < 1$, we further obtain $\sup_{t\in [0,T]}|\tau^h_{\lfloor t/h\rfloor}- t|^{\frac 12 -\frac{\varepsilon}{\lambda}} \le \sup_{k\in\{0,\ldots,\lfloor T/h \rfloor\}}\left|\tau^h_{k}- kh\right|^{\frac 12 -\frac{\varepsilon}{\lambda}} + h^{\frac 12 -\frac{\varepsilon}{\lambda}}$. 
This, together with Proposition~\ref{cor:26112018a1}, implies that
for all $\alpha\in[1,\infty)$ 
there exists a constant $C_6(\alpha,\varepsilon,T)\in (0,\infty)$ such that
\begin{align}\label{eq:time_lag_aux}
\left\|\sup_{t\in [0,T]}|\tau^h_{\lfloor t/h\rfloor}- t|^{\frac 12 -\frac{\varepsilon}{\lambda}} \right\|_{L^\alpha(P_y)}
& \le  h^{\frac{1}{2}-\frac{\varepsilon}{\lambda}}+
\left\|\sup_{k\in \{0,\ldots,\lfloor T/h \rfloor\}}\left|\tau^h_k-
kh\right|^{\frac 12 -\frac{\varepsilon}{\lambda}} \right\|_{L^\alpha(P_y)}\nonumber\\
&\leq C_6(\alpha,\varepsilon,T) h^{\frac{\lambda}{2} - \varepsilon}.
\end{align}
On the set $\{\tau^h_{\lfloor T/h \rfloor}< T+1\}$ it holds that
\begin{align*}
\sup_{t\in [0,T]}|Y_{\tau^h_{\lfloor t/h\rfloor}}- Y_{t}|
&\le\left[\sup_{\substack{s,t\in [0,T+1]\\
s\neq t}}\frac{|Y_s-Y_t|}{|s-t|^{\frac 12 -\frac{\varepsilon}{\lambda}}}\right] \left[ \sup_{t\in [0,T]}|\tau^h_{\lfloor t/h\rfloor}- t|^{\frac 12 -\frac{\varepsilon}{\lambda}}\right].
\end{align*}
This, the Hölder inequality (applied with the conjugates $\frac{1}{p\varepsilon}$ and $\frac{1}{1-p\varepsilon}$), \eqref{eq:mod_cont} and \eqref{eq:time_lag_aux} ensure that there exists a constant $C_7(p,\varepsilon,T)\in (0,\infty)$ such that
\begin{multline}\label{eq:rate_triangle2}
\left\|1_{\{\tau^h_{\lfloor T/h \rfloor} < T+1\}}\sup_{t\in [0,T]}|Y_{\tau^h_{\lfloor t/h\rfloor}}- Y_{t}|\right\|_{L^p(P_y)}\\
\le 
 \left\|\sup_{\substack{s,t\in [0,T+1]\\
s\neq t}}\frac{|Y_s-Y_t|}{|s-t|^{\frac 12 -\frac{\varepsilon}{\lambda}}}\right\|_{L^{\frac{1}{\varepsilon}}(P_y)} \left\|\sup_{t\in [0,T]}|\tau^h_{\lfloor t/h\rfloor}- t|^{\frac 12 -\frac{\varepsilon}{\lambda}}\right\|_{L^{\frac{p}{1-p\varepsilon}}(P_y)}
\le  C_7(p,\varepsilon,T)(1+k_2|y|) h^{\frac \lambda 2 -\varepsilon}.
\end{multline}
Next, we consider the first term on the right-hand side of \eqref{eq:rate_triangle}.
By definition \eqref{eq:07082018a2} of $ Y^h$ it holds that
\begin{equation}
\sup_{t\in [0,T]}| Y^h_{t} - Y_{\tau^h_{\lfloor t/h\rfloor}}|\le 
\sup_{k\in \{0,\ldots,\lfloor T/h \rfloor\}}|Y_{\tau^h_{k+1}} - Y_{\tau^h_{k}}|
=\sup_{k\in \{0,\ldots,\lfloor T/h \rfloor\}}|a_h(Y_{\tau^h_{k}})|.
\end{equation}
Then Proposition~\ref{prop:cond_d_lin_growth},
Theorem~\ref{prop:moment_bound_approx}
and~\eqref{eq:24022019b2}
show that there exist $C_8, C_9(p,T)\in (0,\infty)$ such that
\begin{equation}\label{eq:rate_triangle3}
\begin{split}
\left\|1_{\{\tau^h_{\lfloor T/h \rfloor} < T+1\}}\sup_{t\in [0,T]}| Y^h_{t} - Y_{\tau^h_{\lfloor t/h\rfloor}}|\right\|_{L^p(P_y)}
&\le C_8 h^{\frac{1}{2}}
\left\|
\sup_{k\in \{0,\ldots,\lfloor T/h \rfloor\}}
(1+k_2|Y_{\tau^h_{k}}|)
\right\|_{L^p(P_y)}\\
&\le C_9(p,T)(1+k_2|y|)h^{\frac{1}{2}}.
\end{split}
\end{equation}
Combining \eqref{eq:rate_triangle} with \eqref{eq:rate_triangle2} and \eqref{eq:rate_triangle3} shows that there exists $C_{10}(p,\varepsilon,T)\in (0,\infty)$ such that
\begin{equation}\label{eq:rate_path_aux2}
\left\|1_{\{\tau^h_{\lfloor T/h \rfloor} < T+1\}}\sup_{t\in [0,T]}| Y^h_{t}- Y_{t}|\right\|_{L^p(P_y)}
\le 
 C_{10}(p,\varepsilon,T)(1+k_2|y|) h^{\frac \lambda 2 -\varepsilon}.
\end{equation}
Finally, the triangle inequality,
\eqref{eq:rate_path_aux1} and~\eqref{eq:rate_path_aux2} imply the claim.
\end{proof}

\begin{proof}[Proof of Theorem~\ref{thm:wasserstein_path}]
Throughout the proof let $p\in [1,\infty)$ and $T\in (0,\infty)$.
First note that due to
Theorems \ref{prop:moment_bound_y}
and~\ref{prop:moment_bound_approx}
the law of $Y$ and the law of $X^{h,y}$, $h\in (0,\ol h)$, $y\in I^\circ$,
on the path space $C([0,T],I)$
are elements of $\mathcal M_p(C([0,T],I))$.
It follows from~\eqref{eq:24022019b2}
that the random element
$( Y^h, Y)$ constitutes a coupling between
the laws of $X^{h,y}$ and $Y$ on the path space $C([0,T],I)$.
Therefore, Theorem~\ref{thm:wasserstein_path} follows from Theorem~\ref{thm:rate_path}.
\end{proof}

\begin{proof}[Proof of Theorem~\ref{cor:26022019a1}]
We use the shorthand notations $ X^{h,y}$ and $Y$ for the processes $( X^{h,y}_{t})_{t\in [0,T]}$ and $(Y_t)_{t\in [0,T]}$, respectively.
Theorem~4.1 in \cite{Villani} ensures that
for any $h\in (0,\ol h)$ and $y\in I^\circ$ there exists an optimal coupling $(\xi^h, \zeta)$
on a probability space with a measure $P^*_{h,y}$
between the law of 
$ X^{h,y}$ and the law of $Y$ with respect to the $\mathcal W_2$-distance, i.e., it holds
\begin{align*}
&\Law_{P^*_{h,y}}(\xi^h_t;t\in[0,T])
=
\Law_{P}(X^{h,y}_t;t\in[0,T]),
\\
&\Law_{P^*_{h,y}}(\zeta_t;t\in[0,T])
=
\Law_{P_y}(Y_t;t\in[0,T]),
\\
&\mathcal W_2(P\circ (X^{h,y})^{-1}, P_y \circ (Y)^{-1})
=
\sqrt{E^*_{h,y}\left[\|\xi^h-\zeta\|^2_C\right]}.
\end{align*}
It holds for all $h\in (0,\ol h)$ and $y\in I^\circ$ that 
\begin{equation*}
\begin{split}
&\left|E\left[F(X^{h,y}_{t};\,t\in[0,T])\right]
-E_y\left[F(Y_t;\,t\in[0,T])\right]\right|
=\left|E^*_{h,y}\left[F(\xi^h)-F(\zeta)\right]\right|.
\end{split}
\end{equation*}
This together with \eqref{eq:LipConstPol}, the Cauchy-Schwarz inequality and the triangle inequality
in $L^2$
implies for all $h\in (0,\ol h)$ and $y\in I^\circ$ that
\begin{equation}\label{eq:270219a1}
\begin{split}
&\left|E\left[F(X^{h,y}_{t};\,t\in[0,T])\right]
-E_y\left[F(Y_t;\,t\in[0,T])\right]\right|
\le LE^*_{h,y}\left[\left\{1+(\|\xi^h\|_C\vee\|\zeta\|_C)^\alpha\right\}\|\xi^h-\zeta\|_C\right]\\
&\le 
L \sqrt{E^*_{h,y}\left[\left(1+(\|\xi^h\|_C\vee\|\zeta\|_C)^\alpha\right)^2\right]}
\sqrt{E^*_{h,y}\left[\|\xi^h-\zeta\|^2_C\right]}\\
&\le 
L \left(1+ \sqrt{E^*_{h,y}\left[(\|\xi^h\|_C\vee\|\zeta\|_C)^{2\alpha}\right]}\right)
\sqrt{E^*_{h,y}\left[\|\xi^h-\zeta\|^2_C\right]}.
\end{split}
\end{equation}
It follows from Theorem~\ref{thm:wasserstein_path} that there exists $C_1\in [0,\infty)$ such that for all $h\in (0,\ol h)$ and $y\in I^\circ$
\begin{equation}\label{eq:270219a2}
\sqrt{E^*_{h,y}\left[\|\xi^h-\zeta\|^2_C\right]}=
\mathcal W_2(P\circ (X^{h,y})^{-1}, P_y \circ (Y)^{-1})
\le 
C_1 (1+k_2|y|) h^{\min\{\frac{1}{4},\frac{\lambda}{2}\}-\varepsilon}
.
\end{equation}
For the next step first assume that $\alpha>0$ and let $\beta=\max\{2\alpha,1\}$.
By Jensen's inequality and
Theorems \ref{prop:moment_bound_y}
and~\ref{prop:moment_bound_approx}
there exists $C_2\in [0,\infty)$ such that
we have for all $h\in (0,\ol h)$ and $y\in I^\circ$ 
\begin{equation}\label{eq:270219a3}
\begin{split}
\sqrt{E^*_{h,y}\left[(\|\xi^h\|_C\vee\|\zeta\|_C)^{2\alpha}\right]}
&=
\left(E^*_{h,y}\left[(\|\xi^h\|_C\vee\|\zeta\|_C)^{2\alpha}\right]\right)^\frac{\alpha}{2\alpha}
\le
\left(E^*_{h,y}\left[(\|\xi^h\|_C\vee\|\zeta\|_C)^{\beta}\right]\right)^\frac{\alpha}{\beta}\\
&=
\left\| \|\xi^h\|_C\vee\|\zeta\|_C\right\|^\alpha_{L^{\beta}(P^*_{h,y})}
\le 
\left( 
\left\| \|\xi^h\|_C\right\|_{L^{\beta}(P^*_{h,y})}+\left\|\|\zeta\|_C\right\|_{L^{\beta}(P^*_{h,y})}
\right)^\alpha\\
&=
\left( 
\left\| \|X^h\|_C\right\|_{L^{\beta}(P^*_{h,y})}+\left\|\|Y\|_C\right\|_{L^{\beta}(P^*_{h,y})}
\right)^\alpha
\le 
C_2(1+|y|^\alpha).
\end{split}
\end{equation}
Note that \eqref{eq:270219a3} clearly also holds in the case $\alpha=0$. Combining \eqref{eq:270219a1}, \eqref{eq:270219a2} and \eqref{eq:270219a3} establishes \eqref{eq:rate_path_locLip} and completes the proof.
\end{proof}

\begin{proof}[Proof of Theorem~\ref{th:15042020a1}]
Throughout the proof let $T\in(0,\infty)$,
$y\in I^\circ$ and $p\in[1,\infty)$.
As in the proof of Theorem~\ref{thm:wasserstein_path},
Statement~\eqref{eq:24022019d1}
follows from Theorems \ref{prop:moment_bound_y}
and~\ref{prop:moment_bound_approx}.
Let $F\colon C([0,T],I)\to\bbR$ be a continuous path functional satisfying
$$
|F(x)|\le L(1+\|x\|_C^p),\quad x\in C([0,T],I),
$$
with some constant $L\in[0,\infty)$.
By Theorem~\ref{prop:moment_bound_approx},
we have
$\sup_{h\in(0,\ol h)}E\left[\|X^{h,y}\|_C^{2p}\right]<\infty$,
hence
$$
\sup_{h\in(0,\ol h)}E\left[F(X^{h,y})^2\right]<\infty,
$$
which implies the uniform integrability of the family
$\{F(X^{h,y})\}_{h\in(0,\ol h)}$.
By the main result in \cite{aku2018cointossing},
we have the weak convergence $X^{h,y}\xrightarrow[]{w}Y$ on the path space.
The continuity of the functional $F$ yields
$F(X^{h,y})\xrightarrow[]{w}F(Y)$,
which, together with the uniform integrability of
$\{F(X^{h,y})\}_{h\in(0,\ol h)}$,
implies~\eqref{eq:15042020a3}.
Finally, the equivalence between \eqref{eq:15042020a2} and~\eqref{eq:15042020a3}
follows from Theorem~6.9 in \cite{Villani}
(also consult Definition~6.8 there).
\end{proof}

\section{On Condition~(C)}\label{sec:CC}
In this section we show that Condition~(C) is a nearly minimal  
assumption for our results.
To this end fix any starting point $y\in I^\circ$ and
note that $Y$ is a $P_y$-local martingale
due to our assumption that the boundaries
are inaccessible or absorbing.
Let us consider the conditions
\begin{equation}\label{eq:13072018a1}
r=\infty\quad\text{and}\quad\int^\infty x\,m(dx)<\infty
\end{equation}
and
\begin{equation}\label{eq:13072018a2}
l=-\infty\quad\text{and}\quad\int_{-\infty} |x|\,m(dx)<\infty,
\end{equation}
where the notation
$\int^\infty x\,m(dx)<\infty$
is understood as
$\int_z^\infty x\,m(dx)<\infty$
for some (equivalently, for any) $z\in I^\circ$,
and the notation
$\int_{-\infty} |x|\,m(dx)<\infty$
is understood in a similar way.
The article \cite{Kotani2006} establishes that $Y$ is a
strict $P_y$-local martingale\footnote{A
\emph{strict} local martingale is a local martingale
that fails to be a martingale.}
if and only if at least one of the conditions 
\eqref{eq:13072018a1} and \eqref{eq:13072018a2}
is satisfied.
\cite{GUZ} complements this result by showing that

\smallskip
(a) $Y$ is \emph{not} a $P_y$-supermartingale
if and only if \eqref{eq:13072018a2} holds;

\smallskip
(b) $Y$ is \emph{not} a $P_y$-submartingale
if and only if \eqref{eq:13072018a1} holds.

\smallskip\noindent
Note that under Condition~(C) neither~\eqref{eq:13072018a1} nor~\eqref{eq:13072018a2} is satisfied.

\begin{ex}
Suppose that $r=\infty$ and that the speed measure is given by $m(dx)=\frac{2}{\eta^2(x)}\,dx$ on $I^\circ$ with $\eta$ satisfying $\liminf_{x\to \infty} \frac{\eta^2(x)}{x^2\log^\alpha(x)}>0$ for some $\alpha\in (1,\infty)$. The fact that
$$\lim_{x\to \infty}\int_e^x\frac{z}{z^2(\log(z))^{\alpha}}dz=\frac{1}{\alpha-1}\lim_{x\to \infty}\left(1-\frac{1}{(\log(x))^{\alpha-1}} \right)=\frac{1}{\alpha-1}$$
implies that \eqref{eq:13072018a1} is satisfied. In particular, it follows from \cite{Kotani2006} that the associated diffusion $Y$ is a strict $P_y$-local martingale. 
Hence, there exists $T\in (0,\infty)$ such that
$E_y[\sup_{t\in [0,T]}|Y_t|]=\infty$.
In particular, this implies that the claim of Theorem~\ref{prop:moment_bound_y}
(and hence that of Theorem~\ref{thm:wasserstein_path})
does not hold true in this example, even for $p=1$,
as the law of $Y$ does not belong to
$\cM_1(C([0,T],I))$.
Moreover, also the result on path-dependent rate in $L^p$,
Theorem~\ref{thm:rate_path},
which a priori does not require the law of $Y$ to belong to
$\cM_p(C([0,T],I))$
(only the difference $Y^h-Y$ matters),
does not hold true, even for $p=1$,
as for any fixed $h\in(0,\ol h)$ the law of $Y^h$
belongs to $\cM_1(C([0,T],I))$
(see~\eqref{eq:24022019b2} and \eqref{eq:def_X}--\eqref{eq:13112017a1}).
The error criterion
in the left-hand side of~\eqref{eq:13032019a1}
is just too strong for this example
regardless of the specific approximation scheme used.

Notice that Condition~(C) is only marginally violated
in this example.
\end{ex}

\begin{ex}
Consider any general diffusion $Y$ for which
\begin{equation}\label{eq:13032019a2}
\text{one of
\eqref{eq:13072018a1}--\eqref{eq:13072018a2}
is satisfied, while the other is not.}
\end{equation}
As in the previous example,
it is possible to only marginally violate Condition~(C)
and to achieve~\eqref{eq:13032019a2}.
It follows from Statements (a) and~(b) above
that $Y$ is either a strict $P_y$-supermartingale
or a strict $P_y$-submartingale.\footnote{A
\emph{strict} supermartingale is a supermartingale
that is not a martingale. A \emph{strict} submartingale
is understood in a similar way.}
In particular, there exists $T\in (0,\infty)$ such that
$E_y[Y_T]\ne y$. The approximating Markov chain
$(X^{y,h}_{kh})_{k\in\bbN_0}$
is always a discrete-time martingale
(see~\eqref{eq:def_X} and also notice that
each random variable $X^{y,h}_{kh}$
takes finitely many values and, in particular, belongs to $L^1$).
Therefore, we always have
$E[X^{y,h}_T]=y$. Thus, in this example we have
$$
E[X^{y,h}_T]\not\to E_y[Y_T]\quad\text{as }h\to 0,
$$
and hence, the claim of
Theorem~\ref{th:15042020a1}
does not hold true even for the
(non-path-dependent)
functional $F(x)=x(T)$.
Similarly, the claim of Theorem~\ref{thm:wasserstein_term}
(let alone the ones of Theorems
\ref{thm:wasserstein_path} and~\ref{cor:26022019a1})
are violated as well
whenever \eqref{eq:13032019a2} is satisfied.
\end{ex}

\paragraph{Acknowledgement}
We thank two anonymous referees for the careful reading of the manuscript and for their comments.
Thomas Kruse and Mikhail Urusov acknowledge the support from the
\emph{German Research Foundation}
through the project 415705084.

\bibliographystyle{abbrv}
\bibliography{literature}

\begin{thebibliography}{10}

\bibitem{AJKH}
A.~Alfonsi, B.~Jourdain, and A.~Kohatsu-Higa.
\newblock Pathwise optimal transport bounds between a one-dimensional diffusion
  and its {E}uler scheme.
\newblock {\em Ann. Appl. Probab.}, 24(3):1049--1080, 2014.

\bibitem{AKKK17}
S.~Ankirchner, N.~Kazi-Tani, M.~Klein, and T.~Kruse.
\newblock Stopping with expectation constraints: 3 points suffice.
\newblock {\em Electron. J. Probab.}, 24:Paper No.~66, 16~pp., 2019.

\bibitem{AKKU2018}
S.~Ankirchner, M.~Klein, T.~Kruse, and M.~Urusov.
\newblock On a certain local martingale in a general diffusion setting.
\newblock {\em Preprint, hal-01700656}, 2018.

\bibitem{aklu2020emcel}
S.~Ankirchner, T.~Kruse, W.~L\"ohr, and M.~Urusov.
\newblock Properties of the {EMCEL} scheme for approximating irregular
  diffusions.
\newblock {\em Preprint, arXiv:2004.10316}, 2020.

\bibitem{aku-jmaa}
S.~Ankirchner, T.~Kruse, and M.~Urusov.
\newblock Numerical approximation of irregular {SDE}s via {S}korokhod
  embeddings.
\newblock {\em J. Math. Anal. Appl.}, 440(2):692--715, 2016.

\bibitem{aku2018cointossing}
S.~Ankirchner, T.~Kruse, and M.~Urusov.
\newblock A functional limit theorem for coin tossing {M}arkov chains.
\newblock {\em \textup{Accepted in }Ann. Inst. Henri Poincar\'{e} Probab.
  Stat.}, 2020.

\bibitem{Bass2014}
R.~F. Bass.
\newblock A stochastic differential equation with a sticky point.
\newblock {\em Electron. J. Probab.}, 19:no. 32, 22, 2014.

\bibitem{BSWOHRKK}
C.~{Brugger}, C.~{de Schryver}, N.~{Wehn}, S.~{Omland}, M.~{Hefter},
  K.~{Ritter}, A.~{Kostiuk}, and R.~{Korn}.
\newblock Mixed precision multilevel {M}onte {C}arlo on hybrid computing
  systems.
\newblock In {\em 2014 IEEE Conference on Computational Intelligence for
  Financial Engineering Economics (CIFEr)}, pages 215--222, March 2014.

\bibitem{CanCaglar2019}
B.~Can and M.~Caglar.
\newblock Conditional law and occupation times of two-sided sticky {B}rownian
  motion.
\newblock {\em Preprint, arXiv:1910.10213}, 2019.

\bibitem{EberleZimmer2019}
A.~Eberle and R.~Zimmer.
\newblock Sticky couplings of multidimensional diffusions with different
  drifts.
\newblock {\em Ann. Inst. Henri Poincar\'{e} Probab. Stat.}, 55(4):2370--2394,
  2019.

\bibitem{ep2014}
H.-J. Engelbert and G.~Peskir.
\newblock Stochastic differential equations for sticky {B}rownian motion.
\newblock {\em Stochastics}, 86(6):993--1021, 2014.

\bibitem{ES1985}
H.~J. Engelbert and W.~Schmidt.
\newblock On solutions of one-dimensional stochastic differential equations
  without drift.
\newblock {\em Z. Wahrsch. Verw. Gebiete}, 68(3):287--314, 1985.

\bibitem{EL}
P.~Etor{\'e} and A.~Lejay.
\newblock A {D}onsker theorem to simulate one-dimensional processes with
  measurable coefficients.
\newblock {\em ESAIM: Probability and Statistics}, 11:301--326, 2007.

\bibitem{FGV2016}
T.~Fattler, M.~Grothaus, and R.~Vo\ss{}hall.
\newblock Construction and analysis of a sticky reflected distorted {B}rownian
  motion.
\newblock {\em Ann. Inst. Henri Poincar\'{e} Probab. Stat.}, 52(2):735--762,
  2016.

\bibitem{Frikha:18}
N.~Frikha.
\newblock On the weak approximation of a skew diffusion by an {E}uler-type
  scheme.
\newblock {\em Bernoulli}, 24(3):1653--1691, 2018.

\bibitem{GHMR:19}
M.~B. Giles, M.~Hefter, L.~Mayer, and K.~Ritter.
\newblock Random bit quadrature and approximation of distributions on {H}ilbert
  spaces.
\newblock {\em Found. Comput. Math.}, 19(1):205--238, 2019.

\bibitem{GV2017}
M.~Grothaus and R.~Vo\ss{}hall.
\newblock Stochastic differential equations with sticky reflection and boundary
  diffusion.
\newblock {\em Electron. J. Probab.}, 22:Paper No. 7, 37, 2017.

\bibitem{GV2018}
M.~Grothaus and R.~Vo\ss{}hall.
\newblock Strong {F}eller property of sticky reflected distorted {B}rownian
  motion.
\newblock {\em J. Theoret. Probab.}, 31(2):827--852, 2018.

\bibitem{GUZ}
A.~Gushchin, M.~Urusov, and M.~Zervos.
\newblock On the submartingale/supermartingale property of diffusions in
  natural scale.
\newblock {\em Proc. Steklov Inst. Math.}, 287(1):122--132, 2014.

\bibitem{HajriCaglarArnaudon:17}
H.~Hajri, M.~Caglar, and M.~Arnaudon.
\newblock Application of stochastic flows to the sticky {B}rownian motion
  equation.
\newblock {\em Electron. Commun. Probab.}, 22:Paper No. 3, 10, 2017.

\bibitem{HJK}
M.~Hutzenthaler, A.~Jentzen, and P.~Kloeden.
\newblock Strong and weak divergence in finite time of {E}uler's method for
  stochastic differential equations with non-globally {L}ipschitz continuous
  coefficients.
\newblock {\em Proc. R. Soc. Lond. Ser. A Math. Phys. Eng. Sci.},
  467(2130):1563--1576, 2011.

\bibitem{Kallenberg2002}
O.~Kallenberg.
\newblock {\em Foundations of modern probability}.
\newblock Probability and its Applications (New York). Springer-Verlag, New
  York, second edition, 2002.

\bibitem{KSS2011}
I.~Karatzas, A.~N. Shiryaev, and M.~Shkolnikov.
\newblock On the one-sided {T}anaka equation with drift.
\newblock {\em Electron. Commun. Probab.}, 16:664--677, 2011.

\bibitem{KS}
I.~Karatzas and S.~E. Shreve.
\newblock {\em Brownian {M}otion and {S}tochastic {C}alculus}, volume 113 of
  {\em Graduate Texts in Mathematics}.
\newblock Springer-Verlag, New York, second edition, 1991.

\bibitem{kloeden1992numerical}
P.~Kloeden and E.~Platen.
\newblock {\em Numerical {S}olution of {S}tochastic {D}ifferential
  {E}quations}, volume~23.
\newblock Springer, 1992.

\bibitem{KHLY:JCAM2017}
A.~Kohatsu-Higa, A.~Lejay, and K.~Yasuda.
\newblock Weak rate of convergence of the {E}uler-{M}aruyama scheme for
  stochastic differential equations with non-regular drift.
\newblock {\em J. Comput. Appl. Math.}, 326:138--158, 2017.

\bibitem{KonMen:17}
V.~Konakov and S.~Menozzi.
\newblock Weak error for the {E}uler scheme approximation of diffusions with
  non-smooth coefficients.
\newblock {\em Electronic Journal of Probability}, 22, 2017.

\bibitem{Konarovskyi2017}
V.~Konarovskyi.
\newblock Coalescing-fragmentating {W}asserstein dynamics: particle approach.
\newblock {\em Preprint, arXiv:1711.03011v3}, 2017.

\bibitem{KvR2017}
V.~Konarovskyi and M.~von Renesse.
\newblock Reversible coalescing-fragmentating {W}asserstein dynamics on the
  real line.
\newblock {\em Preprint, arXiv:1709.02839v2}, 2017.

\bibitem{Kotani2006}
S.~Kotani.
\newblock On a condition that one-dimensional diffusion processes are
  martingales.
\newblock In {\em Seminar on probability, XXXIX}, volume 1874 of {\em Lecture
  Notes in Math.}, pages 149--156. Springer, Berlin, 2006.

\bibitem{LLP2019}
A.~Lejay, L.~Len\^{o}tre, and G.~Pichot.
\newblock An exponential timestepping algorithm for diffusion with
  discontinuous coefficients.
\newblock {\em J. Comput. Phys.}, 396:888--904, 2019.

\bibitem{milstein2015uniform}
G.~N. Milstein and J.~Schoenmakers.
\newblock Uniform approximation of the {C}ox-{I}ngersoll-{R}oss process via
  exact simulation at random times.
\newblock {\em Adv. in Appl. Probab.}, 48(4):1095--1116, 2016.

\bibitem{NgoTaguchi:SPL2017}
H.-L. Ngo and D.~Taguchi.
\newblock Strong convergence for the {E}uler-{M}aruyama approximation of
  stochastic differential equations with discontinuous coefficients.
\newblock {\em Statist. Probab. Lett.}, 125:55--63, 2017.

\bibitem{Pages:18}
G.~Pag\`es.
\newblock {\em Numerical probability}.
\newblock Universitext. Springer, Cham, 2018.
\newblock An introduction with applications to finance.

\bibitem{piskorski2016optimal}
T.~Piskorski and M.~M. Westerfield.
\newblock Optimal dynamic contracts with moral hazard and costly monitoring.
\newblock {\em Journal of Economic Theory}, 166:242--281, 2016.

\bibitem{RY}
D.~Revuz and M.~Yor.
\newblock {\em Continuous martingales and {B}rownian motion}, volume 293 of
  {\em Grundlehren der Mathematischen Wissenschaften [Fundamental Principles of
  Mathematical Sciences]}.
\newblock Springer-Verlag, Berlin, third edition, 1999.

\bibitem{RogersWilliams}
L.~C.~G. Rogers and D.~Williams.
\newblock {\em Diffusions, {M}arkov processes, and martingales. {V}ol. 2}.
\newblock Cambridge Mathematical Library. Cambridge University Press,
  Cambridge, 2000.
\newblock It\^o calculus, Reprint of the second (1994) edition.

\bibitem{Villani}
C.~Villani.
\newblock {\em Optimal transport. Old and new}, volume 338 of {\em Grundlehren
  der Mathematischen Wissenschaften [Fundamental Principles of Mathematical
  Sciences]}.
\newblock Springer-Verlag, Berlin, 2009.

\bibitem{zhu2012optimal}
J.~Y. Zhu.
\newblock Optimal contracts with shirking.
\newblock {\em Review of Economic Studies}, 80(2):812--839, 2013.

\end{thebibliography}

\end{document}